\documentclass{article}

\usepackage[all]{xy}
\usepackage[dvips]{graphicx}
\usepackage{amsmath,amssymb}
\usepackage{amsthm}
\newtheorem{theorem}{Theorem}[section]
\newtheorem{prop}[theorem]{Proposition}
\newtheorem{lemma}[theorem]{Lemma}

\newtheorem{cor}[theorem]{Corollary}

\theoremstyle{definition}
\newtheorem{defini}[theorem]{Definition}
\newtheorem{example}[theorem]{Example}
\newtheorem{remark}[theorem]{Remark}

\newcommand{\ep}{\varepsilon}

\newcommand{\RR}{\mathbb R}
\newcommand{\CC}{C^{\infty}}
\newcommand{\ZZ}{\mathbb Z}

\newcommand{\QQ}{\mathbb Q}

\newcommand{\Sign}{{\rm Sign}}

\newcommand{\Crit}{{\rm Crit}}

\newcommand{\grad}{{\rm grad}}

\title{An invariant of rational homology 3-spheres via vector fields}

\author{Tatsuro Shimizu\footnote{shimizu@ms.u-tokyo.ac.jp}}
\begin{document}
\maketitle
\begin{abstract}
We define an invariant of rational homology 3-spheres via vector fields.
The construction of our invariant is a generalization of both that of the Kontsevich-Kuperberg-Thurston invariant and that of
Watanabe's Morse homotopy invariant, which implies the equivalence of these two invariants.
\end{abstract}
\large
\section{Introduction.}\label{intro}
In this paper, we construct an invariant $\widetilde z_n$ of rational homology 3-spheres via vector fields.
As an application, we prove that the Kontsevich-Kuperberg-Thurston invariant $z^{\rm KKT}=\{z_n^{\rm KKT}\}_{n\in\mathbb N}$ coincides with Watanabe's Morse homotopy invariants $z^{\rm FW}=\{z^{\rm FW}_{2n,3n}\}_{n\in \mathbb N}$ for any rational homology 3-sphere.
Note that both $z^{\rm KKT}_n$ and $z_n^{\rm FW}$ are topological invariants which take values in the real vector space $\mathcal A_n(\emptyset)$ of Jacobi diagrams.

M.~Kontsevich~\cite{Kon}, S.~Axelrod and I. M.~Singer~\cite{AS} proposed the Chern-Simons perturbation theory and gave a topological invariant of 3-manifolds. 
Based on Kontsevich' s work, G.~Kuperberg and D.~Thurston constructed in \cite{KKT} a topological invariant $z^{\rm KKT}$ of rational homology 3-spheres. 
Kuperberg and Thurston proved that $z^{\rm KKT}$ is a universal finite type invariant for homology 3-spheres by showing surgery formulas.
C.~Lescop obtained surgery formulas of other types in \cite{Splitting} and \cite{Surgery}.
Lescop reviewed $z^{\rm KKT}$ and gave a more direct proof of well-definedness of this invariant in \cite{Lescop}.

K.~Fukaya~\cite{Fukaya} constructed a topological invariant of 3-manifolds with local coefficients using Morse functions.
Fukaya's invariant is closely related to the theta graph $\theta$. His invariant essentially takes values in $\mathcal A_1(\emptyset)$.
M. Futaki~\cite{Futaki} pointed out that Fukaya's invariant depends on the choice of Morse functions. 
T.~Watanabe~\cite{Watanabe} gave an invariant of rational homology 3-spheres without local coefficients using Morse functions.
He also investigated higher loop graphs (and broken graphs) and then he defined a topological invariant $z_{2n,3n}^{\rm FW}$ of (rational) homology 3-spheres taking values in $\mathcal A_n(\emptyset)$ for each $n\in \mathbb N$.
The construction of $z_{2,3}^{\rm FW}$ is related to the construction of a Morse propagator constructed by Lescop~\cite{LescopHeegaard}.

Fukaya's construction is inspired by the construction of the 2-loop term of the Chern-Simons perturbation theory
and he conjectured in \S 8 in \cite{Fukaya} that his invariant is related to the 2-loop term of the Chern-Simons perturbation theory.
Watanabe also conjectured in Conjecture 1.2 in \cite{Watanabe} that his invariants is related to Axelrod and Singer's invariant~\cite{AS} or Kontsevich's invariant~\cite{Kon}.
 
The main theorem of this paper is the following.
\begin{theorem}\label{theorem1}
$z_n^{\rm KKT}(Y)=z_{2n,3n}^{\rm FW}(Y)$ for any rational homology 3-sphere $Y$, for any $n\in \mathbb N$.
\end{theorem}
The idea of the proof of Theorem~\ref{theorem1} is the following.
We construct an invariant $\widetilde z_n$ of rational homology 3-spheres using vector fields. 
Let $Y$ be a rational homology 3-sphere and let $\infty\in Y$ be a base point.
$z_n^{\rm KKT}(Y)$, $z^{\rm FW}_{2n,3n}(Y)$ and $\widetilde z_n$ are defined by using an extra information of $Y$.
The extra information used in definition of $z_n^{\rm KKT}$, $z_{2n,3n}^{\rm FW}$ and $\widetilde z_n$ are a framing of $Y\setminus\infty$,
a family of Morse functions on $Y\setminus \infty$ and a family of vector fields on $Y\setminus \infty$, respectively.
We prove that it is possible to regard the constructions of $z^{\rm FW}_{2n,3n}$ and $z_n^{\rm KKT}$ as special cases of the construction of $\widetilde z_n$.
In fact a framing gives us a non-vanishing vector field and a Morse function gives us a gradient vector field.
The principal term of $\widetilde z_1$ is related to Lescop's invariant \cite{comb} for rational homology 3-spheres with non-vanishing vector fields.

The organization of this paper is as follows.
In Section 2 we prepare some notations.
In Section 3 we review notions and facts about configuration spaces and graphs discussed by Lescop \cite{Lescop} and Watanabe \cite{Watanabe}.
In Section 4 we define the invariants $\widetilde z_n$ using vector fields and prove the independence of the choice of vector fields.
In Section 5 we review the construction in Lescop \cite{Lescop} of $z^{\rm KKT}$.  
In Section 6 we review the construction of $z^{\rm FW}$ in Watanabe \cite{Watanabe} with a little modification.
In Section 7 we prove Theorem~\ref{theorem1}.
In Section 8 we prove some Lemmas for a compactification of the moduli space of flow graphs used in Sections 6 and 7. 
In Appendix A we give a more direct proof of Theorem~\ref{theorem1} in the case of $n=1$.
\subsection*{Acknowledgments.}
The author would like to thank Professor Mikio Furuta for his encouragement and 
for helpful comments and suggestions in particular about Morse functions 
on punctured manifolds.
The author would also like to thank Professor Tadayuki Watanabe for his
helpful comments and suggestions for an earlier draft and his patient explanation of the 
detail of the construction of his invariant.
The author also expresses his appreciation to Professor Christine Lescop for her
kind and  helpful comments and suggestions to improve an earlier
draft. The last part of the proof of
Lemma~\ref{keylemma} is due to her ideas.
\section{Notation and some remarks.}
In this article, all manifolds are smooth and oriented.
Homology and cohomology are with rational coefficients.
Let $c$ be a $\QQ$-linear sum of finitely many maps from compact $k$-dimensional manifolds with corners to a topological space $X$.  
We consider $c$ as a $k$-chain of $X$ via appropriate  (not unique) triangulations of each $k$-manifold.
Let $Y$ be a submanifold of a manifold $X$.
Let $c=\sum_ia_i(f_i:\Sigma_i\to X)$ be a chain of $X$, where $f_i:\Sigma_i\to X$ are smooth maps from compact manifolds with corners and $a_i$ are rational numbers.
If $f_i$ is transverse to $Y$ for each $i$, then we say that $c$ is transverse to $f$.

When $B$ is a submanifold of a manifold $A$, We denote by $A(B)$ the manifold given by real blowing up of $A$ along $B$.
Namely $A(B)=(A\setminus B)\cup S\nu_{B}$ where $\nu_B$ is the normal bundle of $B\subset A$ and $S\nu_B$ is the sphere bundle of $\nu_B$
(see \cite{KKT} for more details of real blow up).
Note that if a submanifold $C\subset A$ is transverse to $B$, then $C(A\cap B)$ is a proper embedded submanifold of $A(B)$.

Let us denote by $\Delta\subset A\times \cdots \times A$ the fat diagonal of the direct sum of a manifold $A$. 

Let us denote by $\underline{\RR^k}$ the trivial vector bundle over an appropriate base space with rank $k\in \mathbb N$.
For a real vector space $X$, we denote by $SX$ or $S(X)$ the unit sphere of $X$ 
and for a real vector bundle $E\to B$ over a manifold $B$, we denote by $SE$ or $S(E)$ the unit sphere bundle of $E$.   
\subsection{Notations about 3-manifolds and Morse functions.}
Let $f:Y\to \RR$ be a Morse function on a 3-dimensional manifold $Y$ with a metric satisfying the Morse-Smale condition.
Let $\grad f$ be the gradient vector field of $f$ and the metric of $Y$.
Let us denote by $\Crit (f)$ the set of all critical points of $f$.
Let $\{\Phi_f^t\}_{t\in\RR}:Y\to Y$ be the 1-parameter group of diffeomorphisms associated to $-\grad~f$.
We denote by
$$\mathcal A_p=\{x\in Y\mid \lim_{t\to\infty}\Phi_f^t(x)=p\}~{\rm and}$$
$$\mathcal D_p=\{x\in Y\mid \lim_{t\to-\infty}\Phi_f^t(x)=p\}$$
the ascending manifold and descending manifold at $p\in \Crit(f)$ respectively.
\subsection{Conventions on orientations.}
Boundaries are oriented by the outward normal first convention.
Products are oriented by the order of the factors.
Let $y\in B$ be a regular point of a smooth map $f:A\to B$ between smooth manifolds $A$ and $B$.
Let us orient $f^{-1}(y)$ by the following rules: 
$T_xf^{-1}(y)\oplus f^*T_{f(x)}B=T_xA$, for any $x\in f^{-1}(C)$
where $f^*:T_{f(x)}B\to T_xA$ is a linear map satisfying $f_*\circ f^*={\rm id}_{T_{f(x)}B}$.
We denote by $-X$ the orientation reversed manifold of oriented manifold $X$.

Suppose that $Y,f$ and $\grad f$ are as above.
Let us orient ascending manifolds and descending manifolds by imposing the condition:
$T_p\mathcal A_p\oplus T_p\mathcal D_p\cong T_pY$ for any $p\in \Crit(f)$.
Let $p,q\in \Crit(f)$ be the critical points of index $2$ and $1$ respectively.
By the Morse-Smale condition, $\mathcal D_q\cap \mathcal A_p$ is a 1-manifold. 
Let us orient $\mathcal D_q\cap \mathcal A_p$ by the following rule:
$$T_{q'}(\mathcal D_q\cap \mathcal A_p)\oplus T_{q'}\mathcal D_q\cong T_{q'}\mathcal D_p,$$
where $q'\in \mathcal D_q\cap \mathcal A_p$ is a point near $q$.
\begin{figure}[h]
\begin{center}
\includegraphics[width=140pt]{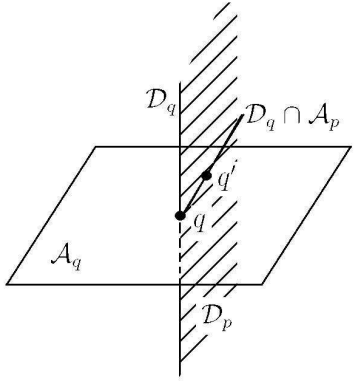} 
\caption{The orientation of $\mathcal D_q\cap \mathcal A_p$.}
\end{center} 
\end{figure}
\section{Configuration space and Jacobi diagrams.}
In this section, we introduce some notations about configuration spaces and Jacobi diagrams.
Most of this section depends on Lescop \cite{Lescop}.
\subsection{The configuration space $C_{2n}(Y)$.}\label{c_2(Y)}
The reference here is Lescop \cite[\S 1.1,1.2,2.1]{Lescop}.

Let $Y$ be a homology 3-sphere with a base point $\infty$.
Let $N(\infty;Y)$ be a regular neighborhood (that is diffeomorphic to an open ball) of $\infty$ in $Y$
and let $N(\infty;S^3)$ be a regular neighborhood of $\infty$ in $S^3=\RR^3\cup \infty$.
We fix a diffeomorphism $\tau^{\infty}:(N(\infty;Y),\infty)\cong (N(\infty;S^3),\infty)$
between $N(\infty;Y)$ and $N(\infty;S^3)$.
We identify $N(\infty;Y)$ with $N(\infty;S^3)$ under $\tau^{\infty}$.

Let $\breve C_{2n}(Y)=(Y\setminus \infty)^{2n}\setminus \Delta=\{\{1,\cdots,2n\}\hookrightarrow Y\setminus \infty\}$
and let $C_{2n}(Y)$ the compactification of $\breve C_{2n}(Y)$ given by Lescop \cite[\S 3]{Lescop}.
(This compactification is similar to Fulton-MacPherson compactification~\cite{FM}).
Roughly speaking, $C_{2n}(Y)$ is obtained from $Y^{2n}$ by real blowing up along all diagonals and 
$\{(x_1,\cdots,x_{2n}\mid \exists i~\mbox{such that}~x_i=\infty\}$. See \S 3 in \cite{Lescop} for the complete definition.)
Note that $C_2(Y)$ is given by real blowing up $Y^2$ along  
$(\infty,\infty)$, $\infty\times (Y\setminus\infty), (Y\setminus \infty)\times \infty$ and $\Delta$ in turn.
Let us denote by $q:C_2(Y)\to (Y\setminus \infty)^2$ the composition of the blow down maps.
Then $\partial C_2(Y)=ST_{\infty}Y\times (Y\setminus \infty)\cup (Y\setminus \infty)\times ST_{\infty}Y\cup S\nu_{\Delta(Y\setminus  \infty)}\cup q^{-1}(\infty^2)$.
We identify $S\nu_{\Delta(Y\setminus \infty)}$ with $STY|_{Y\setminus \infty}$ by the canonical isomorphism $S\nu_{\Delta Y}\cong STY$.
The involution $Y^2\to Y^2, (x,y)\mapsto (y,x)$ induces an involution of $C_2(Y)$.
We denote by $\iota:C_2(Y)\to C_2(Y)$ this involution.

Let
$p_1:(\partial C_2(Y)\supset) ST_{\infty}Y\times (Y\setminus \infty)\to ST_{\infty}Y\stackrel{\tau^{\infty}}{=}ST_{\infty}S^3=S^2$ and
$p_2:(\partial C_2(Y)\supset) (Y\setminus \infty)\times ST_{\infty}Y\to ST_{\infty}Y=ST_{\infty}S^3=S^2$
be the projections.
We denote by $\iota_{S^2}:S^2\to S^2$ the involution induced by $\times (-1):\RR^3\to \RR^3$.

Let $p_c:C_2(S^3)\to S^2$ be the extension of the map ${\rm int}C_2(S^3)=(\RR^3\times \RR^3)\setminus\Delta\to S^2$, $(x,y)\mapsto (y-x)/\|y-x\|$.
Since it is possible to identify $q^{-1}(N(\infty;Y)^2)\subset \partial C_2(Y)$ with $q^{-1}(N(\infty;S^3)^2)\subset \partial C_2(S^3)$ by $\tau^{\infty}$,
we get a map $\partial C_2(Y)\supset q^{-1}((N(\infty;Y)\setminus \infty)^2)\stackrel{p_c}{\to}S^2$.
Since $p_1, \iota_{S^2}\circ p_2$ and $p_c$ are compatible on boundary, 
these maps define the map $$p_Y:\partial C_2(Y)\setminus S\nu_{\Delta(Y\setminus N(\infty;Y))}\to S^2.$$
(Here we note that $\partial C_2(Y)\setminus S\nu_{\Delta(Y\setminus N(\infty;Y))}=ST_{\infty}Y\times (Y\setminus \infty)\cup (Y\setminus \infty)\times ST_{\infty}Y\cup q^{-1}(N(\infty;Y)^2)$.)
\subsection{More on the boundary $\partial C_{2n}(Y)$.}
The reference here is Lescop \cite[\S 2.2, \S 3]{Lescop}.
For $B\subset \{1,\cdots,2n\}$, we set
$$F(\infty;B)=q^{-1}(\{(x_1,\cdots, x_{2n})\mid x_i=\infty~\mbox{iff}~i\in B,~\mbox{if}~i,j\not\in B~\mbox{then}~ x_i\not=x_j  \}) ,$$
and for $B\subset \{1,\cdots,2n\}$($\sharp B\ge 2$), we set
$$F(B)=q^{-1}(\{(x_1,\cdots, x_{2n})\in(Y\setminus \infty)^{2n}\mid\exists y, x_i=y~\mbox{iff}~i\in B,~\mbox{if}~i,j\not\in B~\mbox{then}~ x_i\not=x_j \}).$$
Under these notations, $\partial C_{2n}(Y)=\bigcup_B F(\infty;B)\cup \bigcup_{\sharp B\ge 2}F(B)$.
We remark that $\partial C_{2n}(Y)$ has smooth structure (See \cite[\S~3]{Lescop}). 

Let $X$ be a 3-dimensional real vector space. Let $V$ be a finite set.
we define $\breve S_V(X)$ to be the set of injective maps from $V$ to $X$ up to translations and dilations.
Set $k=\{1,\cdots,k\}$. Note that $\breve S_2(X)=S(X)$.
For an $\RR^3$ vector bundle $E\to M$, we denote by $\breve S_V(E)\to M$ the fiber bundle where the fiber over $x\in M$ is $\breve S_V(E_x)$.
Under these notations, $F(2n)=\breve S_{2n}(T(Y\setminus \infty))$.

We remark that $F(B)$ has a fiber bundle structure where the typical fiber is $\breve S_B(\RR^3)$. 

Lescop gave a compactification $S_V(X), S_V(E)$ of $\breve S_V(X), \breve S_V(E)$ respectively in \cite{Lescop}.
Let $f(B)(X)=\breve S_B(X)\times \breve S_{\{b\}\cup B}(X)$, for $B\subset V$ with $B\not=V$ and  $\sharp B\ge 2$.
Let $f(B)(E)\to M$ be the fiber bundle where the fiber over $x\in M$ is $f(B)(E_x)$.
Under this notation, 
$$\partial S_V(X)=\bigcup_{\sharp B\ge 2}f(B)(X),  \partial S_V(E)=\bigcup_{\sharp B\ge 2}f(B)(E)$$ 
(See Proposition 2.8 in \cite{Lescop}).
We remark that $f(B)(E)$ has a fiber bundle structure where the typical fiber is $\breve S_B(\RR^3)$.  
\subsection{Jacobi diagrams.}
The reference here is Lescop \cite[\S 1.3, 2.3]{Lescop}.
A {\it Jacobi diagram} of degree $n$ is defined to be a trivalent graph with $2n$ vertices and $3n$ edges without simple loops.
For a Jacobi diagram $\overline{\Gamma}$, we denote by $H(\overline \Gamma), E(\overline \Gamma)$ and $V(\overline\Gamma)$ 
the set of half edges, the set of edges and the set of vertices respectively.
An {\it orientation} of a vertex of $\overline{\Gamma}$ is a cyclic order of the three half-edges that meet at the vertex.
A Jacobi diagram is {\it oriented} if all its vertices are oriented.
Let 
$$\mathcal A_n(\emptyset)=\{\mbox{degree}~n~\mbox{oriented Jacobi diagrams}\}\RR/{\rm AS,IHX},$$
where the relations AS and IHX are locally represented by the following pictures.
\begin{figure}[h]
\begin{center}
\includegraphics[width=280pt]{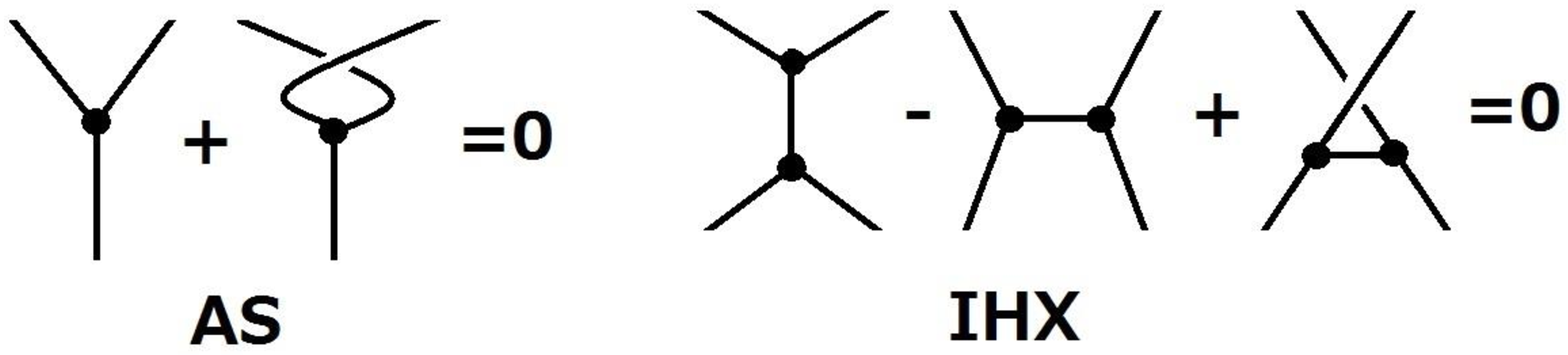} 
\end{center}
Here the orientation of each vertex is given by counterclockwise order of the half edges.
\end{figure}
Let
$$\mathcal E_n=\{\Gamma=(\overline{\Gamma},\varphi_E,\varphi_V,{\rm ori}_E)\}$$
Here $\overline{\Gamma}$ is a Jacobi diagram of degree $n$, $\varphi_E:E(\overline\Gamma)\cong \{1,2,\cdots,3n\}$ and $\varphi_V:V(\overline\Gamma)\cong \{1,2,\cdots,2n\}$ are labels of edges and vertices respectively, and ${\rm ori}_E$ is a collection of orientations of each edge.
These data and an orientation of $\overline{\Gamma}$ induce two orientations of $H(\overline \Gamma)$.
The first one is the edge-orientation induced by $\varphi_E$ and ${\rm ori}_E$.
The second one is the vertex-orientation induced by $\varphi_V$ and orientation of $\overline \Gamma$.
We choose the orientation of $\overline{\Gamma}$ so that the edge-orientation coincides with the vertex-orientation.
Let us denote by $[\Gamma]\in \mathcal A_n(\emptyset)$ the oriented Jacobi diagram given by $\Gamma$ in such a way.
\begin{remark}
The notation $\mathcal A_{2n,3n}$ used by Watanabe \cite{Watanabe} coincides with the notation $\mathcal A_n(\emptyset)$ used by Lescop \cite{Lescop}
as $\RR$-vector spaces.
\end{remark}
\section{Construction of an invariant of rational homology 3-sphere via vector fields.}\label{newinv}
Let $n$ be a natural number.
In this section, we define an invariant $\widetilde z_n$ using vector fields. 
The idea of construction of $\widetilde z_n$ is based on Kuperberg, Thurston \cite{KKT}, Lescop \cite{Lescop} and  
the construction of the anomaly part of Watanabe's Morse homotopy invariant~\cite{Watanabe}.
Let $Y$ be a rational homology 3-sphere with a base point $\infty$.
In Subsection~\ref{ad.form}, we will define the notion of admissible vector fields on $T(Y\setminus \infty)$.
In Subsections~\ref{mainterm},~\ref{anomaly term}, we will define $\widetilde z_n(Y;\vec\gamma)$ and $\widetilde z_n^{\rm anomaly}(\vec\gamma)$
using a family of admissible vector fields $\vec\gamma$.
Thus we obtain a topological invariant $\widetilde z_n(Y)=\widetilde z_n(Y;\vec\gamma)-\widetilde z_n^{\rm anomaly}(\vec\gamma)$ of $Y$ 
in Subsection~\ref{25}.
We will prove well-definedness of $\widetilde z_n$ in Subsection~\ref{s46}.
\subsection{Admissible vector fields on $T(Y\setminus \infty)$.}\label{ad.form}
For $a\in S^2\subset\RR^3$, the map $q_a:\RR^3\to \RR$ is defined by $q_a(x)=\langle x,a\rangle$ where $\langle,\rangle$ is the standard inner product on $\RR^3$.
Write $\pm a=\{a,-a\}$.
\begin{defini}
A vector field $\gamma\in \Gamma T(Y\setminus\infty)$ is an {\it admissible vector field (with respect to $a$)} if the following conditions hold.
\begin{itemize}
\item $\gamma|_{N(\infty;Y)\setminus \infty}=-\grad~q_a|_{N(\infty;S^3)\setminus \infty}$,
\item $\gamma$ is transverse to the zero section in $T(Y\setminus \infty)$.
\end{itemize}
\end{defini}
\begin{example}
We give two important examples of admissible vector fields with respect to $a$.
\begin{itemize}
\item[{\rm (1)}] Let $\tau_{\RR^3}:T\RR^3\cong \underline{\RR^3}$ be the standard framing of $T\RR^3$.
We regard $a\in \RR^3$ as a constant section of the trivial bundle $\underline{\RR^3}$.
For a framing $\tau:T(Y\setminus \infty)\cong \underline{\RR^3}$ such that $\tau|_{N(\infty;Y)\setminus\infty}=\tau_{\RR^3}|_{N(\infty;S^3)\setminus \infty}$, the pull-back vector field $\tau^*a$ is an admissible vector field with respect to $-a$.
\item[{\rm (2)}] For a Morse function $f:Y\setminus\infty\to\RR$ such that $f|_{N(\infty;Y)\setminus\infty}=q_a|_{N(\infty;S)\setminus \infty}$
, $\grad f$ is an admissible vector field with respect to $a$.
\end{itemize}
\end{example}
The following lemma plays an important role in the next subsection.
For an admissible vector field $\gamma$, 
let 
$$\overline c_{\gamma}=\overline{\left\{\frac{\gamma(x)}{\|\gamma(x)\|}\in ST_xY~\Big|~ x\in Y\setminus(\infty\cup \gamma^{-1}(0))\right\}}^{\rm closure}\subset ST(Y\setminus\infty).$$
Here we choose the orientation of $\overline c_{\gamma}$ such that the restriction of the projection $STY\to Y$ to $\overline c_{\gamma}$ is 
orientation preserving. 
\begin{lemma}\label{keylemma}
$$c_0(\gamma)=\overline c_{\gamma}\cup \overline c_{-\gamma}$$
is a submanifold of $ST(Y\setminus \infty)$ without boundary.
\end{lemma}
To prove this lemma, we first remark the following lemma.
Let $n,k\ge 0$ be integral numbers.
Let $s:(\RR^{n+k},0)\to (\RR^{n},0)$ be a $\CC$ map which is transverse to the origin $0\in\RR^{n}$.
\begin{lemma}\label{lemmaA}
There is a diffeomorphism $\varphi:(\RR^{n+k},0)\to (\RR^{n+k},0)$
such that $s\circ \varphi$ coincides with $p_{\RR^n}$ as germs at $0\in \RR^{n+k}$.
Here $p_{\RR^n}:\RR^{n+k}=\RR^n\times \RR^k\to \RR^n$ is the orthogonal projection. 
\end{lemma}
\begin{proof}
This is a consequence of the implicit function theorem.
\end{proof}
\begin{proof}[Proof of Lemma~\ref{keylemma}]
It is sufficient to check this claim near $\gamma^{-1}(0)$.
Let $x\in \gamma^{-1}(0)$.
We fix a trivialization $\psi:T(Y\setminus \infty)|_{U_0}\cong U_0\times \RR^3$ on  a neighborhood $U_0$ of $x$ in $Y$.
By the above Lemma~\ref{lemmaA}, there is a neighborhood $U\subset U_0$ of $x$ and local coordinates $\varphi:\RR^3\cong U$(which is independent of $\psi$) such that $(\varphi^{-1}\times {\rm id})\circ\psi\circ\gamma\circ\varphi:\RR^3\to \RR^3\times \RR^3$ is represented by $(\varphi^{-1}\times {\rm id})\circ\psi\circ\gamma\circ\varphi(x)=(x,x)$.
We fix these local trivialization and coordinates and we write $\gamma$ instead of $(\varphi^{-1}\times {\rm id})\circ\psi\circ\gamma\circ\varphi$.

We first show that $\partial\overline c_{\gamma}\cap STU=-(\partial\overline c_{-\gamma}\cap STU)$ as oriented manifolds.  
Let $D_+=\overline c_{\gamma}\cap STU$ and $D_-=\overline c_{-\gamma}\cap STU$. 
Under the above local coordinates, $D_+=\{(tx,x/\|x\|)\mid x\in S^2, t\in [0,\infty)\}\subset (S^2\times [0,\infty)/(S^2\times 0))\times S^2=\RR^3\times S^2$ and
$D_-=\{(tx,-x/\|x\|)\mid x\in S^2, t\in [0,\infty)\}$.
Both projection $\pi:D_{+}\to \RR^3$ and the projection $\pi:D_{-}\to\RR^3$
are orientation preserving (or reversing). 
Let $g:\RR^3\times S^2\to \RR^3\times S^2$ be the bundle map defined by $(x,v)\mapsto (x,-v)$.
So $g:\partial D_+\cong \partial D_-$ is orientation preserving.
On the other hand, $g|_{\{0\}\times S^2}:\{0\}\times S^2\to \{0\}\times S^2$ is orientation reversing.
Hence, the identity map ${\rm id}:\{0\}\times S^2\to \{0\}\times S^2$ is orientation reversing map as a map between $\partial D_+$ and $\partial D_-$.
Therefore $\partial \overline c_{\gamma}=\partial D_+=-\partial D_-=-\partial \overline c_{-\gamma}$.

We next prove that $c_0(\gamma)\cap STU$ is a submanifold of $STU\cong \RR^3\times S^2$.
Let $p_2:\RR^3\times S^2\to S^2$ be the projection.
For each $v\in S^2$, we have $(p_2|_{c_0(\gamma)})^{-1}(v)=\RR v\times \{v\}\subset \RR^3\times S^2$.
The set $\bigcup_{v\in S^2}\RR v\times \{v\}$ is a submanifold of $\RR^3\times S^2$.
In fact, for any $v_0\in S^2$ and for any sufficiently small neighborhood $B_{v_0}\subset S^2$ of $v_0$ we can take a diffeomorphism 
$$\Phi_{v_0}:(\RR^3\times B_{v_0},\bigcup_{v\in B_{v_0}}\RR v\times \{v\})\stackrel{\cong}{\to}(\RR^3\times B_{v_0}, \RR w_0\times B_{v_0})$$
as follow.\footnote{The author is indebted to Professor Christine Lescop for this construction.}
Here $w_0\in S^2\subset\RR^3$ is a point orthogonal to $v_0$ in $\RR^3$ and $\RR w_0$ is the 1-dimensional vector subspace of $\RR^3$ spanned by $w_0$.
For each $v\in B_{v_0}$, let $m(v,w_0)\in S^2$ be the middle point of the geodesic segment from $v$ to $w_0$.
Let $\rho(v,w_0)\in SO(3)$ be the rotation with axis directed by $m(v,w_0)$ and with angle $\pi$.
So $\rho(v,w_0)$ exchanges $v$ and $w_0$.
Then we can define $\Phi_{v_0}:\RR^3\times B_{v_0}\to \RR^3\times B_{v_0}$ by $\Phi_{v_0}(x,v)=(\rho(v,w_0)(x),v)$
for each $(x,v)\in \RR^3\times B_{v_0}$.
 
Therefore $c_0(\gamma)\cap (\RR^3\times S^2)=\bigcup_{v\in S^2}\RR v\times\{v\}$ is a submanifold of $\RR^3\times S^2$.
\end{proof}
\subsection{The principal term $\widetilde z(Y;\vec\gamma)$.}\label{mainterm}
In this subsection, we define the principal term $\widetilde z(Y;\vec\gamma)$ of the invariant $\widetilde z(Y)$.
We define $$c(\gamma)=p_Y^{-1}(\pm a)\cup c_0(\gamma)\subset \partial C_2(Y).$$ 
By the definition of $\gamma$ and Lemma~\ref{keylemma},
$c(\gamma)$ is a closed 3-manifold.
Therefore $[c(\gamma)]\in H_3(\partial C_2(Y);\mathbb R)$.

Let $\omega_{S^2}^a$ be an anti-symmetric closed 2-form on $S^2$ such that
$\omega_{S^2}^a$ represents the Poincar\'e dual of $[\pm a]$ and the support of $\omega_{S^2}^a$ is concentrated in near $\pm a$.
Let $\omega_{\partial}(\gamma)$ be a closed 2-form on $\partial C_2(Y)$ satisfying the following conditions.
\begin{itemize}
\item $2\omega_{\partial}(\gamma)$ represents the Poincar\'e dual of $[c(\gamma)]$, 
\item The support of $\omega_{\partial}(\gamma)$ is concentrated in near $c(\gamma)$,
\item $\iota^*\omega_{\partial}(\gamma)=-\omega_{\partial}(\gamma)$ and
\item $\omega_{\partial}(\gamma)|_{\partial C_2(Y)\setminus S\nu_{\Delta(Y\setminus N(\infty;Y))}}=\frac{1}{2}p_Y^*\omega_{S^2}^a$.
\end{itemize}
Since $Y$ is a rational homology 3-sphere, the restriction $H^2(C_2(Y);\RR)\to H^2(\partial C_2(Y);\RR)$ is an isomorphism.
Thus there is a closed 2-form $\omega(\gamma)$ on $C_2(Y)$ satisfying the following conditions.
\begin{itemize}
\item $\omega(\gamma)|_{\partial C_2(Y)}=\omega_{\partial}(\gamma)$ and 
\item $\iota^*\omega(\gamma)=-\omega(\gamma)$.
\end{itemize}
\begin{defini}[propagator]
We call $\omega(\gamma)$ a {\it propagator with respect to} $\gamma$.
\end{defini}
Take $a_1,\cdots ,a_{3n}\in S^2$ (we may take, for example, $a_1=\cdots=a_{3n}$).
Let $\gamma_i$ be an admissible vector field with respect to $a_i$ and let 
$\omega(\gamma_i)$ be a propagator with respect to $\gamma_i$ for each $i\in \{1,\cdots,3n\}$.
To simplify notation, we write $\vec\gamma$ instead of $(\gamma_1,\cdots, \gamma_{3n})$. 

For each $\Gamma=(\overline \Gamma, \varphi_E,\varphi_V,{\rm ori}_E)\in \mathcal E_n$ 
and for each $\varphi_E^{-1}(i)\in E(\overline \Gamma)$, let $s(i),t(i)\in \{1,\cdots,2n\}$ denote the labels of the initial vertex and the terminal 
vertex of $\varphi_E^{-1}(i)$ respectively.
The embedding $\{1,2\}\cong\{s(i),t(i)\}\hookrightarrow \{1,\cdots, 2n\}$ induces
the projection $\pi_{\breve C_{2n}(Y)}:\breve C_{2n}(Y)\to\breve C_2(Y)$.
Furthermore it is possible to extend $\pi_{\breve C_{2n}(Y)}$ to $C_{2n}(Y)$ by the definition of $C_{2n}(Y)$. We denote by $P_i(\Gamma):C_{2n}(Y)\to C_2(Y)$ such the extended map 
(see \cite{Lescop}~\S 2.3 for more detail).
\begin{defini}
$$\widetilde z_n(Y;\vec\gamma)=\sum_{\Gamma\in\mathcal E_n}\left(\int_{C_{2n}(Y)}\bigwedge_{i}P_i(\Gamma)^*\omega(\gamma_i)\right)[\Gamma]\in \mathcal A_n(\emptyset).$$
\end{defini}
\begin{remark}
By the above definition, the value $\widetilde z_n(Y;\vec\gamma)$ often depends on the choices of $\omega(\gamma_i)$ even if we fix $\vec\gamma$.
We will prove in Subsection~\ref{s46} that $\widetilde z_n(Y;\vec\gamma)$, however, depends only on the choice of $\vec \gamma$ for generic $\vec\gamma$.  
\end{remark}
\subsection{Alternative description of $\widetilde z_n(Y;\vec\gamma)$.}
In this subsection, we give an alternative description of $\widetilde z_n(Y;\vec\gamma)$ using cohomologies of simplicial complexes with coefficients in $\RR$.
This description will be needed in Section~\ref{proof}.  
The admissible vector field $\gamma_i$ with respect to $a_i$ and the 3-cycle $c(\gamma_i)\subset \partial C_2(Y)$ are as above.
Let $T_{C_2(Y)}$ be the simplicial decomposition of $C_2(Y)$ given by pulling back a simplicial decomposition of $C_2(Y)/\iota$.
So the simplicial decomposition $T_{C_2(Y)}$ is compatible with the action of $\iota$.
By replacing such a simplicial decomposition if necessary, we may assume that each simplex of $T_{C_2(Y)}$ is transverse to $c(\gamma_i)$.
Let $\omega_{\partial}^s(\gamma_i)\in S^2(\partial C_2(Y))$ be the 2-cocycle defined by
$\omega_{\partial}^s(\gamma_i)(\sigma)=\frac{1}{2}\sharp (\sigma\cap c(\gamma_i))$ for each 2-cycle $\sigma$ in $T_{C_2(Y)}|_{\partial C_2(Y)}$.
Thus $\omega_{\partial}^s(\gamma_i)$ is anti-symmetric under the involution $\iota$.
Let $\omega^s(\gamma_i)$ be an extension of $\omega_{\partial}^s(\gamma_i)$ to $C_2(Y)=|T_{C_2(Y)}|$ satisfying the following conditions. 
\begin{itemize}
\item $\omega^s(\gamma_i)|_{\partial C_2(Y)}=\omega^s_{\partial}(\gamma_i)$ and
\item $\iota^*\omega^s(\gamma_i)=-\omega^s(\gamma_i)$.
\end{itemize}
We call it a {\it simplicial propagator}.
Take an appropriate simplicial decomposition of $C_{2n}(Y)$.
Then we have the 2-cocycle $P_i(\Gamma)^*\omega^s(\gamma_i)\in S^{2n}(C_{2n}(Y))$.
By the construction, $\bigwedge_i  P_i(\Gamma)^*\omega^s(\gamma_i)$ is a cocycle in $(C_{2n}(Y),\partial C_{2n}(Y))$.
If necessary we replace the simplicial decompositions with a smaller one, we have the following lemma via the intersection theory.
\begin{lemma}[Alternative description of $\widetilde z_n(Y;\gamma)$]\label{alt}
If $(\bigcap_i P_i(\Gamma)^{-1}{\rm support}(\omega^s(\gamma_i)))\\
\cap\partial C_{2n}(Y)=\emptyset$ for any $\Gamma$,
$$\widetilde z_n(Y;\vec\gamma)=\sum_{\Gamma\in\mathcal E_n}\langle\bigwedge_{i}P_i(\Gamma)^*\omega^s(\gamma_i), [C_{2n}(Y),\partial C_{2n}(Y)]\rangle [\Gamma]
\in\mathcal A_n(\emptyset).$$
Here $ [C_{2n}(Y),\partial C_{2n}(Y)]$ denotes the fundamental homology class and $\langle,\rangle$ denotes the Kronecker product.
\end{lemma}
\subsection{The anomaly term $\widetilde z_n^{\rm anomaly}(\vec\gamma)$.}\label{anomaly term}
In this subsection, we define the anomaly term $\widetilde z_n^{\rm anomaly}(Y;\vec\gamma)$ of the invariant $\widetilde z_n(Y)$.
The idea of the construction of this anomaly term is based on the construction of the anomaly term of Watanabe's invariant~\cite{Watanabe}.
Let $Y$, $\infty$, 
$a_1,\cdots,a_{3n}\in S^2$,
$\gamma_1,\cdots,\gamma_{3n}$ (admissible vector fields with respect to $a_1,\cdots,a_{3n}$ respectively) and
$\omega(\gamma_1),\cdots,\omega(\gamma_{3n})$
be the same as above.
Let $X$ be a connected oriented 4-manifold with $\partial X=Y$ and $\chi (X)=0$.
For example, we can take $X=(T^4\sharp \mathbb CP^2)\setminus B^4$ when $Y=S^3$.
For a framing $\tau'$ of $TY$ or $\underline{\RR}\oplus TY$, we denote by $\sigma_Y(\tau')\in \ZZ$ the signature defect of $\tau'$. 
Let $\tau_{S^3}$ be a framing\footnote{There is such a framing. For example, the Lie framing $\tau_{SU(2)}$ of $S^3=SU(2)$ satisfies 
$\sigma_{S^3}(\tau_{SU(2)})=2$. See R.~Kirby and P.~Melvin~\cite{KM} for more details. We can get $\tau_{S^3}$ by modifying $\tau_{SU(2)}$.} of $TS^3$ satisfying the following two conditions: 
\begin{itemize}
\item $\sigma_{S^3}(\tau_{S^3})=2$,
\item $\tau_{S^3}|_{S^3\setminus N'(\infty;S^3)}=\tau_{\RR^3}|_{S^3\setminus N'(\infty;S^3)}$.
\end{itemize}
Here $N'(\infty;S^3)$ is a neighborhood of $\infty$ smaller than $N(\infty;S^3)$, i.e., $\infty\in N'(\infty;S^3)\subset N(\infty;S^3)$.
\begin{remark}
There is no special meaning in the number "$2$" in the condition $\sigma_{S^3}(\tau_{S^3})=2$.
The anomaly term $\widetilde z^{\rm anomaly}_n(\vec\gamma)$ is independent of the choice 
of $\tau_{S^3}$ even if $\sigma_{S^3}(\tau_{S^3})$ not be $2$. 
We remark that there is no framing $\tau$ on $S^3$ such that $\sigma_{S^3}(\tau)=0$.
\end{remark}
Let $\eta_Y$ be the outward unit vector field of $TY=T(\partial X)\subset TX|_Y$ in $TX$.
Since $\chi (X)=0$, it is possible to extend $\eta_Y$ to a unit vector field of $TX$.
We denote by $\eta_X\in \Gamma TX$ such an extended vector field.
Let $T^vX$ be the normal bundle of $\eta_X$. We remark that $T^vX|_Y=TY$.

The vector field $\tau_{S^3}^*a_i$ of $TY|_{N(\infty;Y)}$ is the pull-back of $a_i\in S^2\subset \RR^3$ along
$\tau_{S^3}|_{N(\infty;Y)}$
\footnote{We sometimes regard a framing as a bundle map to the trivial bundle over a point.}.
Since $\gamma_i|_{Y\setminus N(\infty;Y)}\in \Gamma T(Y\setminus N(\infty;Y))$ and $\tau_{S^3}^*a_i|_{N(\infty;Y)}\in \Gamma TY|_{N(\infty;Y)}$
are compatible, these vector fields define the vector field $\gamma'_i\in \Gamma TY$.
Let $\beta_i\in \Gamma T^vX$ be a vector field of $T^vX$ transverse to the zero section in $T^vX$ satisfying $\beta_i|_{Y}=\gamma'_i$.
By a similar argument of Lemma~\ref{keylemma}, 
$$c_0(\beta_i)=\overline{\left\{\frac{\beta(x)}{\|\beta(x)\|}, \frac{-\beta(x)}{\|\beta(x)\|}\in S(T^vX)_x~\Big|~ x\in X\setminus \beta^{-1}(0)\right\}}^{\rm closure}\subset ST^vX$$
is a submanifold of $ST^vX$ satisfying $\partial c_0(\beta_i)\subset STY$.
Hence $c_0(\beta_i)$ is a cycle of $(ST^vX,\partial ST^vX)$.
Here we choose the orientation of $c_0(\beta_i)$ such that the restriction of the projection $ST^vX\to X$ to $c_0(\beta_i)$ is orientation preserving.

We note that $c_0(\beta_i)$ satisfies $c_0(\beta_i)\cap S\nu_{\Delta(Y\setminus N(\infty;Y))}=c_0(\gamma_i)$.
Let $W(\gamma_i)$ be a closed 2-form on $ST^vX$ satisfying the following conditions.
\begin{itemize}
\item $2W(\gamma_i)$ represents the Poincar\'e dual of $[c_0(\beta_i),\partial c_0(\beta_i)]$, 
\item The support of $W(\gamma_i)$ is concentrated in near $c_0(\beta_i)$,
\item $W(\gamma_i)|_{ST(Y\setminus N(\infty;Y))}=\omega_{\partial}(\gamma_i)|_{S\nu_{\Delta(Y\setminus N(\infty;Y))}}$ and 
\item $W(\gamma_i)|_{STN(\infty;Y)}=\frac{1}{2}\tau_{S^3}^*\omega_{S^2}^{a_i}$.
\end{itemize}
For $i\in \{1,2,\cdots, 3n\}$, let $\phi_i^0(\Gamma):\breve S_{2n}(T^vX)\to S_2(T^vX)$ be the map induced by $\{1,2\}\cong\{s(i),t(i)\}\hookrightarrow \{1,\cdots, 2n\}$. It is possible to extend $\phi_i^0(\Gamma)$ to $S_{2n}(T^vX)$. We denote by $\phi_i(\Gamma):S_{2n}(T^vX)\to S(T^vX)$ such the extended map.
By an argument similar to Proposition 4.17 in \cite{Watanabe}, the following lemma holds.
\begin{lemma}\label{lem48}
There exists $\mu_n\in\mathcal A_n(\emptyset)$ such that
$$-\mu_n\Sign X+\sum_{\Gamma\in\mathcal E_n}\int_{S_{2n}(T^vX)}\bigwedge _i \phi_i(\Gamma)^*W(\gamma_i)[\Gamma]\in \mathcal A_n(\emptyset)$$
does not depend on the choice of $X$, $\beta_i$, and $W(\gamma_i)$.
\end{lemma}
\begin{proof}[Proof of Lemma~\ref{lem48}.]
Let $X$ be a closed 4-manifold with $\Sign X=0$ and $\chi(X)=0$. 
When $X$ is not connected, we assume that the Euler number of each component of $X$ is zero.
Let $\eta_X$ be an unit vector field of $TX$ and let $T^vX$ be the normal bundle of $\eta_X$ in $TX$.
Let $\beta_1,\cdots, \beta_{3n}$ be a family of sections of $T^vX$ that are transverse to the zero section in $T^vX$.
Let $W_i$ be a closed 2-form that represents the Poincar\'e dual of $c_0(\beta_i)$ in $ST^vX$, for $i=1,\cdots,3n$.
By a cobordism argument, it is sufficient to show that $\sum_{\Gamma\in\mathcal E_n}\int_{S_{2n}(T^vX)}\bigwedge_i\phi_i(\Gamma)^*W_i[\Gamma]=0$.

We first prove that there exist an oriented compact 5-manifold $Z$ and there exist unit vector fields $\eta_Z^1,\eta_Z^2\in \Gamma TZ$ such that:
\begin{itemize}
\item $\partial Z=X\sqcup X$,
\item $\eta_Z^1,\eta_Z^2$ are linearly independent at any point in $Z$, i.e., $(\eta_Z^1,\eta_Z^2)$ is a 2-framing of $TZ$,
\item $\eta_Z^1|_{\partial Z}$ is the outward unit vector field of $X=\partial Z$, and
\item $\eta_Z^2|_{\partial Z}=\eta_X\sqcup \eta_X$.
\end{itemize}
Since $\Sign X=0$, there exists a connected compact oriented 5-manifold $Z_0$ such that $\partial Z_0=X$.  
Let $\eta_{Z_0}\in \Gamma TZ_0|_{X}$ be the outward unit vector field of $X=\partial Z_0$. 
By attaching 2-handles along the knots generating $H_1(Z_0;\mathbb Z/2)$ if necessary,
we may assume that $H_1(Z_0;\mathbb Z/2)\cong H^4(Z_0;\partial Z_0;\mathbb Z/2)=0$.
Thus the primary obstruction $o_{Z_0}$ to extending the 2-framing $(\eta_{Z_0}, \eta_X)$ of $TZ_0|_{X}$ into $Z_0$ is in $H^5(Z_0,\partial Z_0;\pi_4(V_{5,2}))
=H^5(Z_0,\partial Z_0;\ZZ/2)$.
Let $Z=Z_0\sharp Z_0$. Then the obstruction to extending the 2-framing $(\eta_{Z_0}\sqcup \eta_{Z_0}, \eta_X\sqcup \eta_X)$ of $TZ|_{X\sqcup X}$ into $Z$ is $o_{Z_0}+o_{Z_0}=0\in H^5(Z,\partial Z;\ZZ/2)$.
So we can take $\eta_Z^1, \eta_Z^2$ satisfying the above conditions.

Let $T^vZ$ be the normal bundle of $\langle \eta^1_Z,\eta^2_Z\rangle$ in $TZ$.
Then $T^vZ$ is a rank 3 sub-bundle of $TZ$ satisfying $T^vZ|_{X}=T^vX$.
Let $\widetilde\beta_i\in \Gamma T^vZ$ be a vector field transverse to the zero section in $T^vZ$ satisfying $T^vZ|_{X}=\beta_i$.
Then $c_0(\widetilde\beta_i)=\overline{\left\{\frac{\widetilde\beta_i(x)}{\|\widetilde\beta_i(x)\|}, \frac{-\widetilde\beta_i(x)}{\|\widetilde\beta_i(x)\|}\in S(T^vZ)_x~\Big|~ x\in Z\setminus \widetilde\beta_i^{-1}(0)\right\}}^{\rm closure}\subset ST^vZ$
is a submanifold of $ST^vZ$ satisfying $\partial c_0(\widetilde\beta_i)=c_0(\beta_i)$.
Let $W(\widetilde\beta_i)$ be a closed 2-form on $ST^vZ$ that represents the Poincar\'e dual of $[c_0(\widetilde\beta_i),\partial c_0(\widetilde\beta_i)]$
and satisfying $W(\widetilde\beta_i)|_{ST^vX}=W_i$.
By Stokes' theorem, we have
\begin{eqnarray*}
0&=&\sum_{\Gamma\in\mathcal E_n}\int_{S_{2n}(T^vZ)}d\left(\bigwedge_i\phi_i(\Gamma)^*W(\widetilde\beta_i)\right)[\Gamma]\\
&=&2\sum_{\Gamma\in\mathcal E_n}\int_{S_{2n}(T^vX)}\bigwedge_i\phi_i(\Gamma)^*W_i[\Gamma]+
\sum_{\Gamma\in\mathcal E_n}\int_{\partial S_{2n}(T^vZ)}\bigwedge_i\phi_i(\Gamma)^*W(\widetilde\beta_i)[\Gamma]\\
&=&2\sum_{\Gamma\in\mathcal E_n}\int_{S_{2n}(T^vX)}\bigwedge_i\phi_i(\Gamma)^*W_i[\Gamma]+
\sum_{\Gamma\in\mathcal E_n}\sum_{2\le \sharp B<2n}\int_{f(B)(T^vZ)}\bigwedge_i\phi_i(\Gamma)^*W(\widetilde\beta_i)[\Gamma]\\
&=&2\sum_{\Gamma\in\mathcal E_n}\int_{S_{2n}(T^vX)}\bigwedge_i\phi_i(\Gamma)^*W_i[\Gamma].
\end{eqnarray*}
The lase equation is given by Lemma~\ref{lem419}.
\end{proof}
Let $\tau_Y$ be a framing of $T(Y\setminus \infty)$ satisfying $\tau_Y|_{N(\infty;Y)\setminus\infty}=\tau_{\RR^3}|_{N(\infty;S^3)\setminus \infty}$.
Then $\tau_Y^*\vec a=(\tau_Y^*a_1,\cdots,\tau_Y^*a_{3n})$ is a family of admissible vector fields.
Let $\tau_Y'=\tau_Y|_{Y\setminus N(\infty;Y)}\cup \tau_{S^3}|_{N(\infty;S^3)}$. So $\tau_Y'$ is a framing of $TY$.
Take $W(\tau_Y^*{a_i})|_{STY}=\frac{1}{2}(\tau'_Y)^*\omega_{S^2}^{a_i}$. 
\begin{lemma}
$\int_{S_{2n}(T^vX)}\bigwedge _i \phi_i(\Gamma)^*W(\tau_Y^*a_i)$ is independent of the choice of $a_1,\cdots, a_{3n}$.
\end{lemma}
\begin{proof}
Let $a_i'$ be an alternative choice of $a_i$ for any $i$.
Let $\widetilde\omega_{S^2}^{i}$ be a closed 2-form on $S^2\times [0,1]$ satisfying 
$\widetilde\omega_{S^2}^{i}|_{S^2\times\{0\}}=\omega_{S^2}^{a_i}$ and 
$\widetilde\omega_{S^2}^{i}|_{S^2\times\{1\}}=\omega_{S^2}^{a_i'}$.
Let $STY\times [0,1]\subset ST^vX$ be the collar of $STY$ such that $STY\times \{0\}=\partial ST^vX$.
We take $W(\tau_Y^*a_i)|_{STY\times [0,1]}=\frac{1}{2}(\tau_Y')^*\widetilde\omega_{S^2}^i$. 
Thus $W(\tau_Y^*a_i)|_{STY\times \{1\}}=W(\tau_Y^*a_i')$.
Since Lemma~\ref{lem48} {\rm (1)}, we have
\begin{eqnarray*}
&&\int_{S_{2n}(T^vX)}\bigwedge_i\phi_i(\Gamma)^*W(\tau_Y^*a_i)-
\int_{S_{2n}(T^vX)}\bigwedge_i\phi_i(\Gamma)^*W(\tau_Y^*a_i')\\
&=&\int_{S_{2n}(TY)\times [0,1]}\bigwedge_i\phi_i(\Gamma)^*W(\tau_Y^*a_i)\\
&=&\frac{1}{2^{3n}}\int_{S_{2n}(TY)\times [0,1]}\bigwedge_i\phi_i(\Gamma)^*(\tau_Y'\times {\rm id})^*\widetilde\omega_{S^2}^i\\
\end{eqnarray*} 
The map $S_{2n}(TY)\times[0,1] \stackrel{\prod_i\phi_i(\Gamma)}{\to}(STY\times [0,1])^{3n}\stackrel{(\tau_Y'\times{\rm id})^{3n}}{\to}(S^2\times [0,1])^{3n}$ factors through $S_{2n}(\RR^3)\times [0,1]$.
Hence we have $((\tau_Y'\times{\rm id})^{3n}\circ\prod_i\phi_i(\Gamma))^*\widetilde\omega_{S^2}^i\in {\rm Im}(\Omega^{6n}(S_{2n}(\RR^3)\times[0,1])\to
\Omega^{6n}(S_{2n}(TY)\times [0,1]))$.
Since $\dim S_{2n}(\RR^3)\times [0,1]=6n-3<6n=\dim \bigwedge_i\phi_i(\Gamma)^*(\tau_Y'\times{\rm id})^*\widetilde\omega_{S^2}^i$,
we have $\int_{S_{2n}(TY)\times [0,1]}\bigwedge_i\phi_i(\Gamma)^*(\tau_Y')^*\widetilde\omega_{S^2}^i=0$.
\end{proof}
Because of the above two lemmas,  
$-\mu_n\Sign X+\sum_{\Gamma\in\mathcal E_n}\int_{S_{2n}(T^vX)}\bigwedge _i \phi_i(\Gamma)^*W(\tau_{\RR^3}^*a_i)[\Gamma]$
is independent of the choice of a 4-manifold $X$ bounded by $S^3$ and a family $a_1,\cdots, a_{3n}$.
We define $$c_n=-\mu_n\Sign X+\sum_{\Gamma\in\mathcal E_n}\int_{S_{2n}(T^vX)}\bigwedge _i \phi_i(\Gamma)^*W(\tau_{\RR^3}^*a_i)[\Gamma]\in \mathcal A_n(\emptyset).$$
\begin{defini}
$$\widetilde{z_n}^{\rm anomaly}(\vec\gamma)=-\mu_n\Sign X+\sum_{\Gamma\in\mathcal E_n}\int_{S_{2n}(T^vX)}\bigwedge _i \phi_i(\Gamma)^*W(\gamma_i)[\Gamma]-c_n\in \mathcal A_n(\emptyset).$$
\end{defini}
\begin{remark}
We will show that $\mu_n=\frac{3}{2}c_n$ in Lemma~\ref{Lem74} and Lemma~\ref{Lem75}.
We can show that $\mu_1=72[\theta]\in \QQ[\theta]=\mathcal A_1(\emptyset)$ by explicit computation (cf. the proof of Proposition~\ref{rationalprop}).
\end{remark}
\subsection{Definition of the invariant.}\label{25}
\begin{theorem}\label{mainprop}
$$\widetilde z_n(Y)=\widetilde z_n(Y;\vec\gamma)-\widetilde z_n^{\rm anomaly}(\vec\gamma)\in \mathcal A_n(\emptyset)$$
does not depend on the choice of $\vec\gamma$. Thus $\widetilde z_n(Y)$ is a topological invariant of $Y$.
\end{theorem}
\begin{defini}
$$\widetilde z_n(Y)=\widetilde z_n(Y;\vec\gamma)-\widetilde z_n^{\rm anomaly}(\vec\gamma)\in \mathcal A_n(\emptyset).$$
\end{defini}
\subsection{Well-definedness of $\widetilde z_n(Y)$ (proof of Theorem~\ref{mainprop}).}\label{s46}
In this section we give the sketch of the proof of well-definedness of $\widetilde z_n(Y)$, i.e., Theorem~\ref{mainprop}.
The proof of well-definedness of $\widetilde z_n$ is almost parallel to that of $z^{\rm KKT}_n$ by Lescop \cite{Lescop}.  

Fix $i\in\{1,\cdots,3n\}$.
For any $j\in \{1,\cdots, 3n\}$, let $a_j'$, $\gamma_j'$, $\beta_j'$, $\omega(\gamma_j')$ and $W(\gamma_j')$ be alternative choices of $a_j$, $\gamma_j$, $\beta_j$, $\omega(\gamma_j)$ and $W(\gamma_j)$ respectively.
Here $a_j'=a_j, \gamma'_j=\gamma_j$, $\omega(\gamma_j')=\omega(\gamma_j)$, $\beta'_j=\beta_j$ and $W(\omega'_j)=W(\omega_j)$ for $j\not=i$.
By the same argument of Proposition 2.15 in \cite{Lescop}, we have the following lemma.
\begin{lemma}
There exists a one-form $\eta_{S^2}\in \Omega^1(S^2)$ such that $d\eta_{S^2}=\omega^{a_j'}_{S^2}-\omega^{a_j}_{S^2}$,
and  a one-form $\eta\in \Omega^1(C_2(Y))$ such that
\begin{itemize}
\item $d\eta=\omega(\gamma_i')-\omega(\gamma_i)$,
\item $\eta|_{\partial C_2(Y)\setminus S\nu_{\Delta(Y\setminus N(\infty;Y))}}=p_Y^*\eta_{S^2}$.
\end{itemize}
\end{lemma} 
Similarly, the following lemma holds.
\begin{lemma}
There exists a one-form $\eta_X\in \Omega^1(ST^vX)$ such that
\begin{itemize}
\item $d\eta_{X}=W(\gamma_i')-W(\gamma_i)$,
\item $\eta_X|_{ST(Y\setminus N(\infty;Y))}=\eta|_{S\nu_{\Delta(Y\setminus N(\infty;Y))}}$,
\item $\eta_X|_{ST^vX|_{N(\infty;Y)}}=\tau_{S^3}^*\eta_{S^2}$.
\end{itemize}
\end{lemma} 
\begin{proof}
Set $\eta_X^0=\eta|_{ST(Y\setminus N(\infty;Y))}\cup \tau_{S^3}^*\eta_{S^2}$.
By the construction of $c_0(\beta_i), c_0(\beta_i')$, we have $[W(\gamma_i)]=[W(\gamma_i')]\in H^2(ST^vX)$ (cf. Lemma~\ref{lemma1}).
Thus there is a one-form $\eta_X^1\in\Omega^1(ST^vY)$ such that $d\eta^1_X=W(\gamma_i)-W(\gamma_i')$.
Since $H^1(ST^vX)=0$, there is a function $\mu_X\in \Omega^0(ST^vY)$ such that $d\mu_X=\eta_X^1|_{ST^vY}-\eta_X^0$.
Let $h:ST^vX\to \RR$ be a $C^{\infty}$ function such that $h\equiv 1$ near $ST^vY(=\partial ST^vX)$ and $h\equiv 0$ far from $ST^vY$.
We can take $\eta_X=\eta_X^1-d(h\mu_X)$ using collar of $ST^vY$ in $ST^vX$.
\end{proof}
Set
\[ \widetilde\omega_j=
  \begin{cases}
   \omega(\gamma_j) (=\omega(\gamma_j'))& j\not=i,\\
   \eta & j=i.
  \end{cases} \]
Set
\[ \widetilde W_j=
  \begin{cases}
   W(\gamma_j) (=W(\gamma_j'))& j\not=i,\\
   \eta_X & j=i.
  \end{cases} \]
By Stokes' theorem,  
\begin{eqnarray*}
&&\int_{C_{2n}(Y)}\bigwedge_jP_j(\Gamma)^*\omega(\gamma_j)-
\int_{C_{2n}(Y)}\bigwedge_jP_j(\Gamma)^*\omega(\gamma'_j)\\
&=&\int_{\partial C_{2n}(Y)} \bigwedge_jP_j(\Gamma)^*\widetilde\omega_j\\
&=&\sum_{F\subset \partial C_{2n}(Y): \mbox{face}}\int_F\bigwedge_jP_j(\Gamma)^*\widetilde\omega_j.
\end{eqnarray*}
\begin{lemma}[{Lescop \cite[Lemma 2.17]{Lescop}}]\label{lem417}
For any non-empty subset $B$ of $2n=\{1,\cdots, 2n\}$, for any $\Gamma\in\mathcal E_n$,
$$\int_{F(\infty;B)}\bigwedge_jP_j(\Gamma)^*\widetilde\omega_j=0.$$
\end{lemma}
\begin{lemma}[{Lescop \cite[Lemma 2.18, 2.19, 2.20, and 2.21]{Lescop}}]\label{lem418}
For any $B\subset \{1,\cdots, 2n\}$ with $\sharp B\ge 2$ and $B\not=\{1,\cdots 2n\}$
$$\sum_{\Gamma\in\mathcal E_n}\left(\int_{F(B)}\bigwedge_jP_j(\Gamma)^*\widetilde\omega_j\right)[\Gamma]=0.$$
\end{lemma}
The proofs of these two lemmas are completely same as the proof in \cite{Lescop}.
The following lemma is proved as Lemma 2.18, 2.19, 2.20, and 2.21 in \cite{Lescop}
(See also the proof of Proposition 2.10 in \cite{Lescop}).   
\begin{lemma}\label{lem419}
For any $B\subset \{1,\cdots,2n\}$ with $2\le \sharp B<2n$, 
\begin{itemize}
\item[{\rm (1)}] $\sum_{\Gamma\in\mathcal E_n}\int_{f(B)(T^vX)}\bigwedge_j\phi_j(\Gamma)^*\widetilde W_j[\Gamma]=0$,
\item[{\rm (2)}] $\sum_{\Gamma\in\mathcal E_n}\int_{f(B)(T^vZ)}\bigwedge_j\phi_j(\Gamma)^*W(\widetilde\beta_j)[\Gamma]=0$
(See the proof of Lemma~\ref{lem48} for the notation $Z, W(\widetilde\beta_j)$).
\end{itemize}
\end{lemma}
By Lemma~\ref{lem417} and Lemma~\ref{lem418},
\begin{eqnarray*}
&&\widetilde z_n(Y;\vec \gamma)-\widetilde z_n(Y;\vec \gamma')\\
&=&\sum_{\Gamma\in\mathcal E_n}\left(\int_{C_{2n}(Y)}\bigwedge_jP_j(\Gamma)^*\omega(\gamma_j)\right)[\Gamma]
-\sum_{\Gamma\in\mathcal E_n}\left(\int_{C_{2n}(Y)}\bigwedge_jP_j(\Gamma)^*\omega(\gamma'_j)\right)[\Gamma]\\
&=&\sum_{\Gamma\in\mathcal E_n}\left(\int_{F(2n)}\bigwedge_jP_j(\Gamma)^*\widetilde\omega_j\right)[\Gamma].
\end{eqnarray*}
Since $F(2n)=\breve S(T(Y\setminus \infty))$, the restriction of $P_j(\Gamma)$ to $F(2n)$ coincides with $\phi_j^0(\Gamma):\breve S_{2n}(T(Y\setminus \infty))\to S\nu_{\Delta(Y\setminus \infty)}\subset \partial C_2(Y)$.
Therefore
\begin{eqnarray*}
&&\sum_{\Gamma\in\mathcal E_n}\int_{F(2n)}\bigwedge_jP_j(\Gamma)^*\widetilde\omega_j[\Gamma]\\
&=&\sum_{\Gamma\in\mathcal E_n}\int_{\breve S_{2n}(T(Y\setminus \infty))}\bigwedge_j\phi^0_j(\Gamma)^*\widetilde\omega_j[\Gamma]\\
&=&\sum_{\Gamma\in\mathcal E_n}\int_{\breve S_{2n}(T(Y\setminus N(\infty;Y)))}\bigwedge_j\phi^0_j(\Gamma)^*\widetilde\omega_j[\Gamma]
+\sum_{\Gamma\in\mathcal E_n}\int_{\breve S_{2n}(T(N(\infty;Y)\setminus\infty))}\bigwedge_j\phi^0_j(\Gamma)^*\widetilde\omega_j[\Gamma]\\
&=&\sum_{\Gamma\in\mathcal E_n}\int_{\breve S_{2n}(T(Y\setminus N(\infty;Y)))}\bigwedge_j\phi^0_j(\Gamma)^*\widetilde\omega_j[\Gamma].
\end{eqnarray*}
The last equation comes from the following lemma.
\begin{lemma}\label{dimreason}
$\sum_{\Gamma\in\mathcal E_n}\int_{\breve S_{2n}(T(N(\infty;Y)\setminus\infty))}\bigwedge_j\phi^0_j(\Gamma)^*\widetilde\omega_j[\Gamma]=0$.
\end{lemma}
\begin{proof}
Since $\breve S_{2n}(T(N(\infty;Y)\setminus\infty))=(N(\infty;Y)\setminus\infty)\times \breve S_{2n}(\RR^3)$ and 
$\widetilde\omega_j|_{ST(N(\infty;Y)\setminus \infty)}=\tau_{S^3}^*\omega_{S^2}$ (or $\tau_{S^3}^*\eta_{S^2}$),
the form $\bigwedge_j\phi^0_j(\Gamma)^*\widetilde\omega_j|_{\breve S_{2n}(T(N(\infty;Y)\setminus \infty))}$ is in the image of the map $(\tau_{S^3})^{3n}\circ \prod_j\phi^0_j(\Gamma)$.
The map $(\tau_{S^3})^{3n}\circ \prod_j\phi^0_j(\Gamma)|_{\breve S_{2n}(T(N(\infty;Y)\setminus\infty))}:\breve S_{2n}(T(N(\infty;Y)\setminus\infty))\to (ST(N(\infty;Y)\setminus\infty))^{3n}\to (S^2)^{3n}$
factors through $\breve S_{2n}(\RR^3)$.
Since $\dim\breve S_{2n}(\RR^3)=6n-4<6n-1=\dim \bigwedge_j\phi^0_j(\Gamma)^*\widetilde\omega_j$,
we have  $\sum_{\Gamma\in\mathcal E_n}\int_{\breve S_{2n}(T^vY|_{N(\infty;Y)})}\bigwedge_j\phi^0_j(\Gamma)^*\widetilde\omega_j[\Gamma]=0$.
\end{proof}
On the other hand, by Stokes' theorem,
\begin{eqnarray*}
&&\widetilde z_n^{\rm anomaly}(\vec\gamma)-\widetilde z_n^{\rm anomaly}(\vec\gamma')\\
&=& \sum_{\Gamma\in\mathcal E_n}\int_{S_{2n}(T^vY)}\bigwedge_j\phi_j(\Gamma)^*\widetilde W_j[\Gamma]
+\sum_{\Gamma\in\mathcal E_n}\int_{\partial S_{2n}(T^vX)}\bigwedge_j\phi_j(\Gamma)^*\widetilde{W}_j[\Gamma]\\
&\stackrel{(*)}{=}& \sum_{\Gamma\in\mathcal E_n}\int_{S_{2n}(T^vY)}\bigwedge_j\phi_j(\Gamma)^*\widetilde W_j[\Gamma]\\
&=& \sum_{\Gamma\in\mathcal E_n}\int_{S_{2n}(T(Y\setminus N(\infty;Y))}\bigwedge_j\phi_j(\Gamma)^*\widetilde W_j[\Gamma]+
\sum_{\Gamma\in\mathcal E_n}\int_{S_{2n}(T^vY|_{N(\infty;Y)})}\bigwedge_j\phi_j(\Gamma)^*\widetilde W_j[\Gamma]\\
&=& \sum_{\Gamma\in\mathcal E_n}\int_{S_{2n}(T(Y\setminus N(\infty;Y))}\bigwedge_j\phi_j(\Gamma)^*\widetilde W_j[\Gamma].
\end{eqnarray*}
The equation (*) is given by Lemma~\ref{lem419}{\rm (1)} and 
the last equation comes from the following lemma.
\begin{lemma}
$\sum_{\Gamma\in\mathcal E_n}\int_{S_{2n}(T^vY|_{N(\infty;Y)})}\bigwedge_j\phi_j(\Gamma)^*\widetilde W(\gamma_j')[\Gamma]=0$.
\end{lemma}
The proof of this lemma is parallel to the proof of Lemma~\ref{dimreason}.

Since $\widetilde W_j|_{S\nu_{\Delta(Y\setminus N(\infty;Y))}}=\widetilde\omega_j|_{S\nu_{\Delta(Y\setminus N(\infty;Y))}}$
 for any $j$, we have 
$$\widetilde z_n(Y;\vec \gamma)-\widetilde z_n(Y;\vec \gamma')=
\sum_{\Gamma\in\mathcal E_n}\int_{\breve S_{2n}(T(Y\setminus N(\infty;Y)))}\bigwedge_j\phi^0_j(\Gamma)^*\widetilde\omega_j[\Gamma]~~~~~~~~~~$$
$$~~~~~~~~~~~~~~~=\sum_{\Gamma\in\mathcal E_n}\int_{S_{2n}(T(Y\setminus N(\infty;Y)))}\bigwedge_j\phi_j(\Gamma)^*\widetilde W_j[\Gamma]
=\widetilde z_n^{\rm anomaly}(\vec\gamma)-\widetilde z_n^{\rm anomaly}(\vec\gamma').$$
Now we finish the proof of Theorem~\ref{mainprop}.
\section{Review of $z_n^{\rm KKT}$.}\label{KKT}
In this section, we review the construction of $z_n^{\rm KKT}$ for rational homology 3-spheres.
This section is based on Lescop \cite{Lescop}.

Let $\tau_Y:T(Y\setminus \infty)\cong \underline{\RR^3}$ be a framing satisfying $\tau_Y|_{N(\infty;Y)\setminus\infty}=\tau_{\RR^3}$.
$\tau_Y|_{Y\setminus N(\infty;Y)}\cup \tau_{S^3}|_{N(\infty;S^3)}$ is a framing of $TY$ by the assumption of $\tau_Y$.
We define
$$\sigma_{Y\setminus\infty}(\tau_Y)=\sigma_Y(\tau_Y|_{Y\setminus N(\infty;Y)}\cup \tau_{S^3}|_{N(\infty;S^3)})-\sigma_{S^3}(\tau_{S^3})$$
$$~~~~=\sigma_Y(\tau_Y|_{Y\setminus N(\infty;Y)}\cup \tau_{S^3}|_{N(\infty;S^3)})-2$$
 and call it the {\it signature defect} of $\tau_Y$ of a framing of
$Y\setminus \infty$. For example $\sigma_{\RR^3}(\tau_{\RR^3})=0$.

The canonical isomorphism $S\nu_{\Delta(Y\setminus \infty)}\cong T(Y\setminus \infty)$ and the framing $\tau_Y$ induce the map 
$p_{\Delta}(\tau_Y):S\nu_{\Delta(Y\setminus \infty)}\to S^2$.
Since the assumption of $\tau_Y$, maps $p_{\Delta}(\tau_Y)$ and $p_Y:\partial C_2(Y)\setminus S\nu\to S^2$ are compatible.
So we get the map $p(\tau_Y)=p_Y\cup p_{\Delta}(\tau_Y):\partial C_2(Y)\to S^2$.
Let $\omega_{S^2}\in\Omega^2(S^2)$ be an anti-symmetric 2-form satisfying $\int_{S^2}\omega_{S^2}=1$.
Let $\omega(\tau_Y)$ be an anti-symmetric closed 2-from on $C_2(Y)$ satisfying 
$\omega(\tau_Y)|_{\partial C_2(Y)}=p(\tau_Y)^*\omega_{S^2}\in \Omega^2(\partial C_2(Y))$.
\begin{prop}[{Lescop \cite[Theorem 1.9 and Proposition 2.11]{Lescop}}]
There exists constants $\delta_n\in\mathcal A_n(\emptyset)$ such that
$$\sum_{\Gamma\in \mathcal E_n}\int_{C_{2n}(Y)}
\left(\bigwedge_iP_i(\Gamma)^*\omega(\tau_Y)\right)[\Gamma]-\frac{\sigma_{Y\setminus \infty}(\tau_Y)}{4}\delta_n\in \mathcal A_n(\emptyset)$$
does not depend on the choice of $\tau_Y$. 
\end{prop}
\begin{defini}[{Kuperberg and Thurston~\cite{KKT}, Lescop~\cite{Lescop}}]
$$z_n^{\rm KKT}(Y;\tau_Y)=\sum_{\Gamma\in \mathcal E_n}\int_{C_{2n}(Y)}
\left(\bigwedge_iP_i(\Gamma)^*\omega(\tau_Y)\right)[\Gamma],$$
$$z_n^{\rm KKT}(Y)=z_n^{\rm KKT}(Y;\tau_Y)-\frac{\sigma_{Y\setminus \infty}(\tau_Y)}{4}\delta_n\in \mathcal A_n(\emptyset).$$
\end{defini}
We remark that $\delta_n$ is given by explicit formula in Proposition 2.10 in \cite{Lescop}.  
\begin{remark}
The universal finite type invariant $Z_n^{\rm KKT}$ described in \cite{Lescop} equals to the degree $n$ part of
$\exp(\sum_n \frac{1}{2^{3n}(3n)!(2n)!}z^{\rm KKT}_n)$.
See before Lemma 2.12 in \cite{Lescop} for more detail.
\end{remark}
\begin{remark}
We will show that $\delta_n=\frac{4}{3}\mu_n$ in Lemma~\ref{Lem74}.
\end{remark}
\section{Review of Watanabe's Morse homotopy invariants $z^{\rm FW}_n$.}\label{Morse}
In this section we give a modified construction of Watanabe's Morse homotopy invariant \cite{Watanabe} $z_{2n,3n}^{\rm FW}$ for rational homology 3-spheres.
We will remark the differences between our modified construction and Watanabe's original construction after the definition of $z_{2n,3n}^{\rm FW}(Y)$.  
The invariant $z^{\rm FW}_{2n,3n}(Y)$ is a sum of the principal term $z_{2n,3n}^{\rm FW}(Y;\vec f)$ and the anomaly term
$z_{2n,3n}^{\rm anomaly}(\vec f)$ of $\vec f$
where $\vec f=(f_1,f_2,\cdots,f_{3n})$ is a family of Morse functions on $Y\setminus \infty$.

Fix a point $a\in S^2$.
\begin{defini}
A Morse function $f:Y\setminus \infty\to\RR$ is an {\it admissible Morse function with respect to} $a$ if it satisfies the following conditions.
\begin{itemize}
\item $f|_{N(\infty;Y)\setminus \infty}=q_a|_{N(\infty;S^3)\setminus \infty}$ and
\item $f$ has no critical point of index $0$ or $3$.
\end{itemize}
\end{defini}
Let $\Crit (f)=\{p_1,\cdots,p_k,q_1,\cdots,q_k\}$ be the set of critical points of $f$ where ${\rm ind}(p_i)=2,{\rm ind}(q_i)=1$.
Let $$0\to C_2(Y\setminus \infty;f)\stackrel{\partial}{\to} C_1(Y\setminus \infty;f)\to 0$$
be the Morse complex of $f$ with rational coefficients.
Let $g:C_1(Y\setminus \infty;f)\to C_2(Y\setminus \infty;f)$, $g([q_i])=\sum_jg_{ij}[p_j]$ be the inverse map of the boundary map 
$\partial :C_2(Y\setminus \infty;f)\to C_1(Y\setminus \infty;f), \partial[p_i]=\sum_j\partial_{ij}[q_j]$.
($g$ is called a combinatorial propagator in \cite{Watanabe}.)

We now construct $\mathcal M(f)$ which is the weighted sum of (non-compact) 4-manifold in $Y^2\setminus \Delta$.
Let $M_{\to}(f)={\rm pr}(\varphi^{-1}(\Delta))$ where $\varphi: Y\times Y\times (0,\infty) \to Y\times Y$ is the map defined by $(x,y)\mapsto (y,\Phi^t_f(x))$
and ${\rm pr}:Y\times Y\times (0,\infty)\to Y\times Y$ is the projection.
We choose the orientation of $M_{\to}(f)$ such that the inclusion $Y\times (0,\ep)\hookrightarrow M_{\to}(f), (x,t)\mapsto (x,\Phi_f^t(x))$
preserves orientations.
We define $$\mathcal M(f)=M_{\to}(f)-\sum_{i,j}g_{ij}(\mathcal A_{q_i}\times \mathcal D_{p_j})\setminus \Delta.$$
We remark that the orientation of $\mathcal M(f)$ does not depend on the choice of orientations of $\mathcal A_{q_i},\mathcal D_{p_j}$.

Let $a_1,\cdots ,a_{3n}\in S^2\subset \RR^3$ be the points such that any different three points of them are linearly independent in $\RR^3$.
Let $f_i:Y\setminus \infty\to \RR$ be a sufficiently generic admissible Morse function with respect to $a_i$ for each $i=1,\cdots,3n$.
We write $\vec f=(f_1,\cdots,f_{3n})$ to simplify notation.
We replace a metric of $Y$ such that the Morse-Smale condition holds for each $f_i$ if necessary.

Set $\mathcal M(\pm f_i)=\mathcal M(f_i)+\mathcal M(-f_i)$.
\begin{defini}\label{defini:pt}
For generic $\vec f$,
$$z_{2n,3n}^{\rm FW}(Y;\vec f)=\sum_{\Gamma\in \mathcal E_n}
\frac{1}{2^{3n}}\sharp\left( \bigcap_{i=1}^{3n}P_i(\Gamma)|_{(Y\setminus \infty)^{2n}\setminus \Delta}^{-1}(\mathcal M(\pm f_i))
\right) [\Gamma]\in\mathcal A_n(\emptyset).$$
\end{defini}
We next define the anomaly part.
Set $\grad~\vec f=(\grad~f_1,\cdots,\grad~f_{3n})$.
\begin{defini}
$$z_{2n,3n}^{\rm anomaly}(\vec f)=\tilde z^{\rm anomaly}_n(\grad~\vec f).$$
\end{defini}
\begin{defini}[{Watanabe~\cite{Watanabe}}]
$$z^{\rm FW}_{2n,3n}(Y)=z^{\rm FW}_{2n,3n}(Y;\vec f)-z_{2n,3n}^{\rm anomaly}(\vec f).$$
\end{defini}
\begin{remark}\label{differ}
A difference between our modified construction of $z_{2n,3n}^{\rm FW}$ and
Watanabe's original construction in \cite{Watanabe} is the conditions for Morse functions. 
Our Morse function is on $Y\setminus \infty$ and explicitly written on $N(\infty;Y)\setminus \infty$. 
On the other hand, Watanabe uses any Morse functions on $Y$.
We note that $Y\setminus\infty\subset Y\sharp S^3$ where $Y\sharp S^3$ is the connected sum of $Y$ and $S^3$ at $\infty\in Y$ and $0\in S^3$.
Then it is possible to extend $f:Y\setminus\infty\to \RR$ to $Y\sharp S^3\cong Y$ in standard way.
\begin{figure}[h]
\begin{center}
\includegraphics[width=240pt]{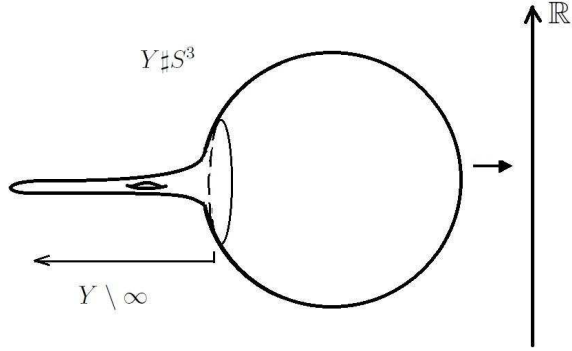} 
\caption{The extension of $f$ to $Y\sharp S^3$}
\end{center}
\end{figure}
Then we can show that the difference between the value $z_{2n,3n}^{\rm FW}(Y)$ described in this Section and
the value of Watanabe's original invariant of $Y$ is a constant which is independent of $Y$.
\end{remark}
We must prove that $\sharp\left( \bigcap_{i}P_i(\Gamma)|_{(Y\setminus \infty)^{2n}\setminus \Delta}^{-1}(\mathcal M(\pm f_i))
\right)$ is well defined for generic $\vec f$, because Morse functions used in the above definition differ from Morse functions used in the original definition in \cite{Watanabe} near $N(\infty;Y)\setminus\infty$ (See Remark~\ref{differ} for more details).
\begin{lemma}
$P_1(\Gamma)|_{(Y\setminus \infty)^{2n}\setminus \Delta}^{-1}(\mathcal M(\pm f_i))$
, $\cdots, P_{3n}(\Gamma)|_{(Y\setminus \infty)^{2n}\setminus \Delta}^{-1}(\mathcal M(\pm f_i))$
transversally intersect at finitely many points, for generic $f_1,\cdots,f_{3n}$ and $a_1,\cdots, a_{3n}$, for any $\Gamma\in \mathcal E_n$.
\end{lemma}
\begin{proof}
Let $x=(x_1,\cdots,x_{2n})\in \bigcap_{i}P_i(\Gamma)|_{(Y\setminus \infty)^{2n}\setminus \Delta}^{-1}(\mathcal M(\pm f_i))
\subset (Y\setminus \infty)^{2n}\setminus \Delta$.
\vskip2mm
\underline{The case of $x\in (Y\setminus N(\infty;Y))^{2n}$.}\\
Thanks to \S 2.4 of \cite{Watanabe}, the transversality at $x$ is given by generic $\vec f$ .
\vskip2mm
\underline{The case of $x\not\in (Y\setminus N(\infty;Y))^{2n}$.}\\
We show that for generic $a_1,\cdots, a_{3n}$, there are no such $x$. (Then, in particular, $\bigcap_{i}P_i(\Gamma)|_{(Y\setminus \infty)^{2n}\setminus \Delta}^{-1}(\mathcal M(\pm f_i))$ is a $0$-dimensional compact manifold).
Let $B=\{ i\in \{1,\cdots,2n\}\mid x_i\in Y\setminus N(\infty;Y)\}$. Let
$$E_B=\{i\in \{1,\cdots,3n\}\cong E(\Gamma)\mid \{s(i),t(i)\}\subset B\},$$
$$E^{\partial}_B=\{i\in \{1,\cdots,3n\}\cong E(\Gamma)\mid \{s(i),t(i)\}\cap B\not=\emptyset\}\setminus E_B.$$
Let $\Gamma/B$ be the labelled graph obtained from $\Gamma$ by collapsing $B$ to a point $b_0$ and removing all edges in $E_B$.
Here the label of edges and vertices of $\Gamma/B$ are $\{1,\cdots, 3n\}\setminus E_B$, $\{0,1,\cdots,2n\}\setminus B$ respectively (the label of $b_0$ is 0).
Note that $\sharp(V(\Gamma/B)-\{b_0\})=2n-\sharp B$ and $\sharp E(\Gamma/B)\geq 3n-\frac{3\sharp B}{2}$.

Let $\pi:Y\setminus \infty\to Y/(Y\setminus N(\infty;Y))\stackrel{\tau_{\infty}}{=}\RR^3$ be the map obtained by collapsing $Y\setminus N(\infty;Y)$ to the point $0\in\RR^3$.
Let $\pi'_i:\RR\to \RR$ be the map obtained by collapsing ${\rm Im}(f_i:Y\setminus N(\infty;Y)\to \RR)$ to $0\in\RR$.
Then $\pi'_i\circ f_i=q_{a_i}\circ\pi:Y\setminus\infty\to \RR$.
Let $x':V(\Gamma/B)-\{b_0\}\hookrightarrow \RR^3$ be the restriction of $\pi\circ x:V(\Gamma)\hookrightarrow \RR^3$ to $V(\Gamma/B)-\{b_0\}\subset V(\Gamma)$.
Let $a'\in (S^2)^{E(\Gamma/B)}$ be the points obtained from $a=(a_1,\cdots,a_{3n})$ removing all $a_i$, $i\in E_B$.
We define the map
$$\varphi:(\RR^3)^{V(\Gamma/B)-\{b_0\}}\setminus\Delta\to (S^2)^{E(\Gamma/B)}$$
as $$\varphi(y)=\left(\frac{y_{s(i)}-y_{t(i)}}{\| y_{s(i)}-y_{t(i)}\|}\right)_{i\in E(\Gamma/B)}.$$
Here if $i\in E^{\partial}_B$ then either $s(i)$ or $t(i)$ is $0$.
Then $x'\in \varphi^{-1}(a')$.
By the following lemma, there is no $x'$ for a generic $a'$.
Therefore there is no $x$ for a generic $a$.
\begin{lemma}
For a generic $a'$ we have $\varphi^{-1}(a')=\emptyset$.
\end{lemma}
\begin{proof}
For any $y\in\varphi^{-1}(a')$ and for any $t\in (0,\infty)$, we have $ty\in\varphi^{-1}(a')$.
Thus if $\varphi^{^1}(a')\not=\emptyset$, we have $\dim \varphi^{-1}(a')\ge 1$.
On the other hand, $\dim ((\RR^3)^{V(\Gamma/B)-\{b_0\}})=6n-3\sharp B\le 2\sharp E(\Gamma/B)=\dim((S^2)^{E(\Gamma/B)})$.
Hence we have $\dim \varphi^{-1}(a')\le 0$ for a generic $a'$. This is contradiction.
\end{proof}
\end{proof}
\section{Proof of Theorem~\ref{theorem1}.}\label{proof}
In this section we prove Theorem~\ref{theorem1} in Section~\ref{intro}.
\subsection{Proof of $\widetilde{z_n}(Y)=z_n^{\rm KKT}(Y)$.}
We follow the notations used in Section~\ref{KKT}. For example, $Y$ is a rational homology 3-sphere and $\infty\in Y$ is  a base point, and so on.
Let $\tau_Y:T(Y\setminus \infty)\cong \underline{\RR^3}$ be a framing of $Y\setminus \infty$ satisfying $\tau_Y|_{N(\infty;Y)\setminus \infty}=\tau_{\RR^3}|_{N(\infty;S^3)\setminus \infty}$.
We denote $\tau_Y^*\vec a=(\tau_Y^*a,\cdots,\tau_Y^*a)$ for $a\in S^2$.
We take $\omega_{S^2}=\frac{1}{2}\omega_{S^2}^a$ in the definition of $z_n^{\rm KKT}(Y;\tau_Y)$, 
and we take $\omega(\tau_Y^*a)=\omega(\tau_Y)$ in the definition of $\widetilde z_n(Y;\tau_Y^*\vec a)$.
Thus 
$$\widetilde z_n(Y;\tau_Y^*\vec a)=\sum_{\Gamma\in\mathcal E_n}\int_{C_{2n}(Y)}\bigwedge_iP_i(\Gamma)^*\omega(\tau_Y)[\Gamma]
=z_n^{\rm KKT}(Y;\tau_Y).$$
Then we only need show that
$$\widetilde z_n^{\rm anomaly}(Y;\tau_Y^*\vec a)=\frac{1}{4}\sigma_{Y\setminus \infty}(\tau_Y)\delta_n$$
in this condition.    

The idea of the proof of $\widetilde z_n^{\rm anomaly}(Y;\tau_Y^*\vec a)=\frac{1}{4}\sigma_{Y\setminus \infty}(\tau_Y)\delta_n$
is as follows. We first prove this equation in the case of $Y=S^3$.
The well-definedness of $\widetilde z_n^{\rm anomaly}(Y)$ implies that $\widetilde z_n^{\rm anomaly}(S^3;\tau^*\vec a)=\frac{1}{4}\sigma_{\RR^3}(\tau)\delta_n$ for any framing $\tau$ of $S^3\setminus\infty$. 
The general case is reduced to the case of $Y=S^3$ by a cobordism argument.  

We introduce notation.
For a compact 4-manifold $X$ such that $\partial X=Y$ and $\chi(X)=0$,
we denote $\widetilde z^{\rm anomaly}(\vec\gamma;X)=\sum_{\Gamma\in\mathcal E_n}\int_{S_{2n}(T^vX)}\bigwedge_i\phi_i(\Gamma)^*W(\gamma_i)[\Gamma]=\widetilde z_n^{\rm anomaly}(\vec \gamma)+\mu_n\Sign X +c_n$.
Then $\widetilde z^{\rm anomaly}(\vec\gamma)=\widetilde z^{\rm anomaly}(\vec\gamma;X)-\mu_n\Sign X-c_n$ by the definition.
\begin{lemma}
$\widetilde z_n(S^3)=z^{\rm KKT}(S^3)$.
\end{lemma}
\begin{proof}
Let $X$ be a compact 4-manifold with $\partial X=S^3$ and $\chi(X)=0$.
\begin{eqnarray*}
\widetilde z_n(S^3)&=&\widetilde z_n(S^3;\tau_{\RR^3}^*\vec a)-\widetilde z_n^{\rm anomaly}(\tau_{\RR^3}^*\vec a;X)+\mu_n\Sign X+c_n\\
&=&\widetilde z_n(S^3;\tau_{\RR^3}^*\vec a)\\
&=&z_n^{\rm KKT}(S^3;\tau_{\RR^3}).
\end{eqnarray*}
Since $\sigma_{\RR^3}(\tau_{\RR^3})=0$,
we have $z_n^{\rm KKT}(S^3;\tau_{\RR^3})=z_n^{\rm KKT}(S^3)$.

Therefore $\widetilde z_n(S^3)=z_n^{\rm KKT}(S^3;\tau_{\RR^3})=z_n^{\rm KKT}(S^3)$.
\end{proof}
Since $\widetilde z_n^{\rm anomaly}(S^3)$ is independent of the choice of framing on $\RR^3=S^3\setminus \infty$,
we have the following corollary. 
\begin{cor}\label{Lem711}
For any framing $\tau$ on $\RR^3=S^3\setminus\infty$ such that $\tau|_{N(\infty;S^3)\setminus \infty}=\tau_{\RR^3}|_{N(\infty;S^3)\setminus \infty}$,
the equation $\widetilde z_n^{\rm anomaly}(S^3;\tau^*\vec a)=\frac{1}{4}\sigma_{\RR^3}(\tau)\delta_n$ holds.
\end{cor}
Recall that the framing $\tau_Y$ of $T(Y\setminus \infty)$ gives the framing $\tau_Y\cup \tau_{S^3}=\tau_Y|_{Y\setminus N(\infty;Y)}\cup \tau_{S^3}|_{N(\infty;S^3)}$ of $TY$ and $\sigma_{Y\setminus\infty}(\tau_Y)=\sigma_{Y}(\tau_Y\cup\tau_{S^3})-\sigma(\tau_{S^3})
=\sigma_{Y}(\tau_Y\cup \tau_{S^3})-2$.
We give the spin structure on $Y$ using $\tau_Y\cup \tau_{S^3}$.
\begin{lemma}
There exists a positive integer $k$ and a spin 4-manifold $X_0$ such that $\chi(X_0)=0$ and $\partial X_0=Y\sqcup k(-S^3)$ as spin manifolds.
Here $-S^3$ is $S^3$ with the opposite orientation.  
\end{lemma} 
\begin{proof}
Since the 3-dimensional spin cobordism group equals to zero, there exists a spin 4-manifold $\widetilde X$ such that $\partial \widetilde X=Y$.
Let $k=\chi(\widetilde X)$.
We may assume that $k\ge 0$, by replacing $\widetilde X$ by $\widetilde X\sharp nK3$ for sufficiently large integer $n$ if necessary.
Let $X_0$ be the spin 4-manifold obtained by removing $k$ disjoint 4-balls, i.e., $X_0=\widetilde X\setminus kB^4$.
Then $\chi (X_0)=0$ and $\partial X_0=Y\sqcup k(-S^3)$.
\end{proof}
\begin{remark}
Since $\chi(X_0\sharp T^4)=\chi(X_0)-2, \chi(X_0\sharp K3)=\chi(X_0)+22$ and $T^4$, $K3$ are spin,
it is possible to choose $k+2n$ instead of $k$ for any $n\in\ZZ$.
\end{remark}
\begin{remark}
Since the Euler number of a closed spin 4-manifold is even, the number $k(Y)=k \mod 2\in\ZZ/2$ is an invariant of a spin 3-manifold $Y$.
It is known that $k(Y)={\rm rk} H_1(Y;\ZZ/2)+1$ (See Theorem 2.6 in \cite{KM}).
We also remark that $k(Y)\equiv \sigma_{Y\setminus\infty}(\tau_Y)+1 \mod 2$.
\end{remark}
Let $X_0$ be a spin 4-manifold such that $\chi(X_0)=0$ and $\partial X_0=Y\sqcup k(-S^3)$ for some $k\ge 1$.
We denote $S^3_i$ the $i$-th $S^3$-boundary of $X_0$.
Then $\partial X_0=Y\sqcup -S^3_1\sqcup\cdots\sqcup -S^3_k$.
By the obstruction theory, it is possible to extend the framing $\eta_Y\oplus(\tau_Y\cup\tau_{S^3})$ of $TX_0|_{Y}$ to $X_0$
 where $\eta_Y$ is the outward unit vector field on $Y\subset \partial X_0$
(see \cite{KM} for more details).
We choose such a extended framing $\widetilde\tau_X$ such that $\widetilde\tau_X^*\null^t(1,0,0,0)|_{k(-S^3)}$ is the inward unit vector field on $k(-S^3)\subset \partial X_0\subset X_0$.
If necessary we modify $\widetilde\tau_X$ by using homotopy, we may assume that there exists a framing $\tau_i$ of $S^3_i\setminus \infty$ 
such that $\tau_i|_{N(\infty;S^3_i)\setminus \infty}=\tau_{\RR^3}|_{N(\infty;S^3)\setminus \infty}$ and $-\eta_i\oplus(\tau_i\cup \tau_{S^3})=\widetilde\tau_X|_{-S^3_i}$. Here $-\eta_i$ is the inward unit vector field on $-S^3_i\subset X_0$.

Let $X'$ be a compact oriented 4-manifold with $\chi(X')=0$ and $\partial X'=S^3$. 
Then $X_0\cup kX'$ is a compact 4-manifold with $\chi(X_0\cup kX')=0$ and $\partial(X_0\cup kX')=Y$.
\begin{lemma}\label{Lem73}
The following three equations hold.
\begin{itemize}
\item[{\rm (1)}] $\widetilde z^{\rm anomaly}_n(\tau_Y^*\vec a;X_0\cup kX')
=\sum_{i=1}^k\widetilde z^{\rm anomaly}_n(\tau_i^*\vec a;X')$.
\item[{\rm (2)}] $\sigma_{Y\setminus \infty}(\tau_Y)=\sum_{i=1}^k\sigma_{\RR^3}(\tau_i)+2(k-1)-3\Sign X_0$.
\item[{\rm (3)}] $\widetilde z^{\rm anomaly}_n(\tau_Y^*\vec a)=\frac{1}{4}\sigma_{Y\setminus \infty}(\tau_Y)\delta_n+\left(\frac{3}{4}\delta_n-\mu_n\right)\Sign X_0+\frac{k-1}{2}\delta_n+(k-1)c_n$.
\end{itemize}
\end{lemma}
\begin{proof}
{\rm (1)} We take a 3-bundle $T^v(X_0\sqcup kX')\subset T(X_0\sqcup kX')$ over $X_0\cup kX'$ such that $T^v(X_0\sqcup kX')|_{X_0}$ is the normal bundle
of $\widetilde\tau_X^*\null^t(1,0,0,0)$.
We denote $T^vX_0=T^v(X_0\sqcup kX')|_{X_0}$, $T^v(kX')=T^v(X_0\sqcup kX')|_{kX'}$.
Let $\beta$ be a section of $T^v(X_0\sqcup kX')$ such that $\beta|_{X_0}=\widetilde\tau_X^*a$ and $\beta$ is transverse to the zero section in $T^v(X_0\sqcup kX')$.
In this setting, we can take $W(\tau_Y^*a)|_{ST^vX_0}=\widetilde\tau_X^*\omega_{S^2}$. 
Then $\widetilde z_n^{\rm anomaly}(\tau_Y^*\vec a;X_0\sqcup kX')=
\sum_{\Gamma\in\mathcal E_n}\int_{S_{2n}(T^v(X_0\sqcup kX'))}\bigwedge_i \phi_i(\Gamma)^*W(\tau_Y^*a)[\Gamma]
=\sum_{\Gamma}\int_{S_{2n}(T^vX_0)}\bigwedge_i\phi_i(\Gamma)^*\widetilde\tau_X^*\omega_{S^2}[\Gamma]\\
+\sum_{\Gamma}\int_{S_{2n}(T^v(kX'))}\bigwedge_i \phi_i(\Gamma)^*W(\tau_Y^*a)[\Gamma]$.

We show that $\int_{S_{2n}(T^vX_0)}\bigwedge_i \phi_i(\Gamma)^*\widetilde\tau_X^*\omega_{S^2}=0$ for any $\Gamma\in\mathcal E_n$.
The map $(\widetilde\tau_X)^{3n}\circ(\prod_i \phi_i(\Gamma)):S_{2n}(T^vX_0)\to (S^2)^{3n}$ factors through 
$S_{2n}(\RR^3)$:
\[\xymatrix{S_{2n}(T^vX_0) \ar[d]_{\widetilde\tau_X} \ar[r]^(0.5){\prod_i\phi_i(\Gamma)} \ar@{}[dr]^(0.4)\circlearrowleft & (ST^vX_0)^{3n} \ar[d]_{(\widetilde\tau_X)^{3n}}  & \\
S_{2n}(\RR^3) \ar[r] & (S^2)^{3n}. &\\
}\]\
Hence we have $\bigwedge_i\phi_i(\Gamma)^*\widetilde\tau_X^*\omega_{S^2}|_{ST^vX_0}=((\prod \widetilde\tau_X)^{3n}\circ \bigwedge_i\phi_i(\Gamma))^*(\omega_{S^2})^{3n}\\
\in {\rm Im}(\Omega^{6n}(S_{2n}(\RR^3))\to\Omega^{6n}(S_{2n}(T^vX_0)))$.
Since $\dim \breve S_{2n}(\RR^3)=6n-4<6n=\dim \bigwedge_i\phi_i(\Gamma)^*\widetilde\tau_X^*\omega_{S^2}$, we have
$\bigwedge_i \phi_i(\Gamma)^*\widetilde\tau_X^*\omega_{S^2}=0$.

Therefore 
\begin{eqnarray*}
\widetilde z_n^{\rm anomaly}(\tau_Y^*\vec a;X_0\sqcup kX')
&=& \sum_{\Gamma\in\mathcal E_n}\int_{S_{2n}(T^v kX')}\bigwedge_i \phi_i(\Gamma)^*W(\tau_Y^*a)[\Gamma]\\
&=&\sum_{i=1}^k\widetilde z_n^{\rm anomaly}(\tau_i^*\vec a;X').
\end{eqnarray*}

{\rm (2)} By the obstruction theory and the definition of the signature defect, we have 
$\sigma_Y(\tau_Y\cup\tau_{S^3})+3\Sign X_0=\sum_{i=1}^k\sigma_{S^3}(\tau_i\cup \tau_{S^3})$.
Since $\sigma_{Y\setminus \infty}(\tau_Y)=\sigma_Y(\tau_Y\cup\tau_{S^3})-2$ and
$\sigma_{\RR^3}(\tau_i)=\sigma_{S^3}(\tau_i\cup\tau_{S^3})-2$, the equation $(2)$ holds.

{\rm (3)} 
\begin{eqnarray*}
\widetilde z_n^{\rm anomaly}(\tau_Y^*\vec a)&&=
\widetilde z_n^{\rm anomaly}(\tau_Y^*\vec a;X_0\sqcup kX')-\mu_n\Sign (X_0\sqcup kX')-c_n\\
&&\stackrel{{\rm (1)}}{=}\widetilde z_n^{\rm anomaly}(\tau_1^*\vec a;X')-\mu_n\Sign X'-c_n\\
&&+\cdots+\widetilde z_n^{\rm anomaly}(\tau_k^*\vec a;X')-\mu_n\Sign X'-c_n\\
&&-\mu_n\Sign X_0+(k-1)c_n\\
&&=\sum_i\widetilde z_n^{\rm anomaly}(\tau_i^*\vec a)-\mu_n\Sign X_0+(k-1)c_n\\
&&\stackrel{{\rm Corollary}~\ref{Lem711}}{=}\sum_i\frac{1}{4}\sigma_{\RR^3}(\tau_i)\delta_n-\mu_n\Sign X_0+(k-1)c_n\\
&&\stackrel{{\rm (2)}}{=}\frac{1}{4}(\sigma_{Y\setminus\infty}(\tau_Y)-2(k-1)+3\Sign X_0)\delta_n-\mu_n\Sign X_0+(k-1)c_n. 
\end{eqnarray*}
\end{proof}
We next compute $\mu_n,c_n$ and prove that $\widetilde z_n^{\rm anomaly}(\tau_Y^*\vec a)=\frac{1}{4}\sigma_{Y\setminus\infty}(\tau_Y)\delta_n$ by using the above lemma.
\begin{lemma}\label{Lem74}
$\mu_n=\frac{3}{4}\delta_n$.
\end{lemma}
\begin{proof}
Let $X_0=K3\sharp 11T^4\setminus (B^4\sqcup B^4)$. 
Then $X_0$ is a spin 4-manifold satisfying $\chi(X_0)=0$ and $\Sign X_0=16$.
It is possible to deal with $\partial X_0=S^3\sqcup -S^3$.
By Lemma~\ref{Lem73} (3), we have 
$0=\widetilde z_n^{\rm anomaly}(\tau_{\RR^3}^*\vec a)
=(\frac{3}{4}\delta_n-\mu_n)\Sign X_0$.
Since $\Sign X_0=16\not=0$, we have $\mu_n=\frac{3}{4}\delta_n$.    
\end{proof}
\begin{lemma}\label{Lem75}
$c_n=\frac{1}{2}\delta_n$.
\end{lemma}
\begin{proof}
Let $X_0=K3\sharp 10T^4\setminus (B^4\sqcup 3B^4)$. 
Then $X_0$ is a spin 4-manifold satisfying $\chi(X_0)=0$ and $\Sign X_0=16$.
It is possible to deal with $\partial X_0=S^3\sqcup 3(-S^3)$.
By Lemma~\ref{Lem73} (3) and Lemma~\ref{Lem74}, we have 
$0=\widetilde z_n^{\rm anomaly}(\tau_{\RR^3}^*\vec a)
=-\delta_n+2c_n$.
Then $c_n=\frac{1}{2}\delta_n$.    
\end{proof}
\begin{prop}\label{77}
$\widetilde z_n^{\rm anomaly}(\tau_Y^*\vec a)=\frac{1}{4}\sigma_{Y\setminus \infty}(\tau_Y)\delta_n.$
\end{prop}
\begin{proof}
Take $X_0, k, \widetilde\tau_X$ as in Lemma~\ref{Lem73}.
By Lemma~\ref{Lem73} (3), Lemma~\ref{Lem74} and Lemma~\ref{Lem75}, we have
$\widetilde z_n^{\rm anomaly}(\tau_Y^*\vec a)=\frac{1}{4}\sigma_{Y\setminus \infty}(\tau_Y)\delta_n-\frac{k-1}{2}\delta_n+(k-1)c_n
=\frac{1}{4}\sigma_{Y\setminus \infty}(\tau_Y)\delta_n$. 
\end{proof}
\subsection{Proof of $\widetilde z_n(Y)=z^{\rm FW}_{2n,3n}(Y)$.}
Let $f$ be an admissible Morse function with respect to $a\in S^2$.
The weighted sum $\mathcal M(f)+\mathcal M(-f)$ consists of weighted pairs of two distinct points on a gradient trajectory. 
There is a compactification $\mathcal M_S(\pm f)$ of $\mathcal M(f)+\mathcal M(-f)$ by adding pairs of points on broken trajectories as the Morse theory. Then $\mathcal M_S(\pm f)$ becomes a 4-cycle in $(C_2(Y),\partial C_2(Y))$ (Lemma~\ref{64}). 
See for Section~\ref{sect6} for the detail of the above argument.
\begin{lemma}
$\partial \mathcal M_S(\pm f)=c(\grad f)$ for any admissible Morse function $f$.
\end{lemma}
\begin{proof}
Since $\grad f|_{N(\infty;Y)}=\grad q_a$, 
if $(x,u)\in \partial \mathcal M_S(\pm f)\cap((Y\setminus \infty)\times ST_{\infty}Y)$ then $u=\pm a$.
On the other hand, $\partial\mathcal M_S(\pm f)\cap(\{x\}\times ST_{\infty}Y)=\{(x,a),(x,-a)\}$ for any $x\not\in \Crit(f)$.
Since $\partial \mathcal M_S(\pm f)$ is a 3-cycle, we have $\partial \mathcal M_S(f)\cap((Y\setminus \infty)\times ST_{\infty}Y)=(Y\setminus\infty)\times (\pm a)$.
With a similar argument, we have $\partial C_2(Y)\setminus S\nu_{\Delta(Y\setminus \infty)}=p_Y^{-1}(\pm a)$.
Since this fact and Lemma~\ref{65} we conclude the proof. 
\end{proof}
We follow the notations $a_1,\cdots, a_{3n}$, $f_1,\cdots, f_{3n}$ as in Section~\ref{Morse}.
In the following proposition, the notion "generic $\vec f$" means that $\bigcap_iP_i(\Gamma)^{-1}\mathcal M_S(\pm f_i)=\emptyset$ for any $\Gamma\in\mathcal E_n$.
We remark that there exists such a $\vec f$ (See Remark~\ref{rmk79}). 
\begin{prop}
For generic $\vec f$,
$z_{2n,3n}^{\rm FW}(Y;\vec f)=
\widetilde z_n(Y;\grad \vec f)$.
\end{prop}
\begin{proof}
We define the 2-cocycle $\omega^s_i(\grad~f_i)\in S^2(|T_{C_2(Y)}|)$ by
$\omega^s(\grad~f_i)(\sigma)=\frac{1}{2}\sharp(\sigma\cap \mathcal M_S(f_i))$ for each 2-cycle $\sigma$ of $T_{C_2(Y)}$. 

By the construction, $\omega^s(\grad~f_i)$ is simplicial propagator for each $i$.
By the intersection theory and Lemma~\ref{alt}, we have
$$z^{\rm FW}_{2n,3n}(Y;\vec f)=\langle \bigwedge_i P_i(\Gamma)^*\omega^s(\grad~f_i),[C_{2n}(Y),\partial C_{2n}(Y)]\rangle
=\frac{1}{2^{3n}}\sharp \left(\bigcap_i P_i(\Gamma)^{-1}\mathcal M_S(f_i)\right) $$
for any $\Gamma\in\mathcal E_n$.
\end{proof}
\begin{remark}\label{rmk79}
We can show that $\partial C_{2n}(Y)\cap (\bigcap_iP_i(\Gamma)^{-1}\mathcal M_S(\pm f_i))=\emptyset$ for generic $\vec f$ by an argument similar to Lemma 2.7 in Watanabe \cite{Watanabe}.
For example,
we take the following $\Phi'_{\Gamma}$ instead of $\Phi$ in Lemma 2.7 in \cite{Watanabe} when we prove  $F(\{1,2,4\})\cap(\bigcap_{i=1}^6P_i({\rm Smooth}(\Gamma))^{-1}\mathcal M_S(\pm f_i))=\emptyset$ for the graph $\Gamma$ in the picture (2.2) in \cite{Watanabe} (See Example 2.6 in \cite{Watanabe} and see \S 3.4 of \cite{Watanabe} for the definition of the operator $\rm Smooth$).  
$$\phi'_{\Gamma}:F(\{1,2,4\})\times \left(\bigcup_{f_1\in\mathcal U_1} \mathcal A_p(f_1)\cap \mathcal D_q(f_1)\right) \times (\RR_{>0})^3\times \prod_{i=2}^4\mathcal U_i~~~~~~~~~~~~~~$$
$$~~~~~~~~~~~~~~~~~~~~~~~~~~~\to Y^3\times (TY)^2\times (TY)^2\times Y^3,$$
\begin{eqnarray*}
&&\Phi'_{\Gamma}(((x_1,[w_1,w_2,w_4]), x_3),u,t_2,t_3,t_4,f_2,f_3,f_4)\\
&&=((x_1,u,\Phi_{f_6}^{t_6}(x_3)), (\grad_{x_1}f_2, \frac{w_2-w_1}{\|w_2-w_1\|}),\\
&&~~(\grad_{x_1}f_3, \frac{w_4-w_2}{\|w_4-w_2\|}), (x_3,\Phi_{f_4}^{t_4}(x_1),\Phi_{f_5}^{t_5}(x_1))).
\end{eqnarray*}
Here $x_1\in Y\setminus \infty$, $[w_1,w_2,w_3]\in \breve S_{\{1,2,4\}}T_{x_1}Y$, $x_3\in Y\setminus \{x_1,\infty\}$.
Let 
$$\Delta'_{\Gamma}=\{((y_1,y_1,y_1), ((y_2,s_2v_2),(y_2,t_2v_2)), ((y_3,s_3v_3),(y_3,t_3v_3)),(y_4,y_4,y_4))$$
$$~~~~~\mid (y_1,y_2,y_3,y_4)\in (Y\setminus \infty)^4, t_i,s_i\ge 0, v_i\in T_{y_i}Y\}.$$
Then $\Phi'_{\Gamma}$ is transverse to $\Delta'_{\Gamma}$ as Lemma~2.7 in \cite{Watanabe}. 
\end{remark}
It is obvious that
 $z_{2n,3n}^{\rm anomaly}(Y;\vec f)=\widetilde z_n^{\rm anomaly}(Y;\grad\vec f)$ by the definitions of the anomaly parts.
\section{Compactification of moduli space $\mathcal M(f)$}\label{sect6}
In this section we give a compactification $\mathcal M_S(\pm f)$ of 
$\mathcal M(f)\cup \mathcal M(-f)$ and then show that $\mathcal M_S(\pm f)$ is a 4-cycle in $(C_2(Y),\partial C_2(Y))$.
Let $M_{\to}(f)=\varphi^{-1}|_{Y^2\times (0,\infty)}(\Delta)$ where $\varphi:Y^2\times (-\infty,\infty)\to Y^2, (x,y)\mapsto (y,\Phi_f^t(x))$.
\begin{lemma}[{Watanabe \cite[Proposition 2.12]{Watanabe} (cf. \cite{BH})}]\label{Lem71}
There is a manifold with corners $\overline{M}_{\to}(f)$ satisfying the following conditions.
\begin{itemize}
\item[{\rm (1)}] $\overline{M}_{\to}(f)=\{g:I\to Y\mid I\subset -\RR,\\
g~\mbox{\rm is a piecewise smooth map}, f(g(t))=t, \frac{dg(t)}{dt}=\frac{\grad_{g(t)} f}{\|\grad_{g(t)}f\|^2}~\mbox{\rm for any}~ t \}$ as sets,
\item[{\rm (2)}] ${\rm int}\overline M_{\to}(f)=M_{\to}(f)$, and
\item[{\rm (3)}] $\partial \overline M_{\to}(f)=\sum_{i}\mathcal A_{p_i}\times \mathcal D_{p_i}+\sum_{j}\mathcal A_{q_j}\times \mathcal D_{q_j}$.
\end{itemize}
\end{lemma}
Note that ${\rm int}(\overline M_{\to}(f)+\overline M_{\to}(-f))=\varphi^{-1}(\Delta)$.
We denote by $\overline M_{\to}(f)\to (Y\setminus \infty)^2$ the continuous map that is the extension of the embedding 
$M_{\to}(f)\to (Y\setminus \infty)^2$ to $\overline M_{\to}(f)$.
For simplicity of notation, we write $\overline M_{\to}(f)$ instead of $\overline{M}_{\to}(f)\to (Y\setminus\infty)^2$.

Similarly we denote by $\overline{\mathcal A_{p_i}}\to Y$ the extension of $B^1(1)\cong \mathcal A_{p_i}\to Y$ to $\overline{B^1(1)}$
and we write $\overline{\mathcal A_{p_i}}$ instead of $\overline{\mathcal A_{p_i}}\to Y$
(We remark that $\mathcal A_{p_i}$ is diffeomorphic to $B^1(1)$ the interior of unit disk in $\RR^1$).
We also define $\overline{\mathcal D_{p_i}}, \overline{\mathcal A_{q_j}}$, and so on.
\begin{lemma}
\begin{itemize}
\item[{\rm (1)}]  $\overline{M}_{\to}(f)+\overline{M}_{\to}(-f)$ is transverse to $\Delta$.
\item[{\rm (2)}] $\overline{\mathcal A_{q_j}}\times \overline{\mathcal D_{p_i}}$ is transverse to $\Delta$.
\end{itemize}
\end{lemma}
\begin{proof}
{\rm (1)} 
$\grad f$(which is the section of $\nu_{\Delta(Y\setminus \infty)}$) is transverse to the zero section in $\nu_{\Delta(Y\setminus \infty)}$.
$\mathcal A_p\times \mathcal D_p\subset Y^2$ is transverse to $\Delta$ for any critical point $p\in \Crit(f)=\Crit(-f)$.
Thanks to Lemma~\ref{Lem71}~{\rm (2),(3)}, this finishes the proof of {\rm (1)}.   

{\rm (2)} is immediate from the Morse-Smale condition.
\end{proof}
By this Lemma, $ (\overline{M}_{\to}(f)+\overline{M}_{\to}(-f))(\Delta)$ and
$(\overline{\mathcal A_{q_j}}\times \overline{\mathcal D_{p_i}})(\Delta)$ are well-defined.
It is clear that $ (\overline{M}_{\to}(f)+\overline{M}_{\to}(-f))(\Delta)=\\
(\overline{M}_{\to}(f)+\overline{M}_{\to}(-f))\setminus \Delta\cup\{(x,\frac{\pm\grad_xf}{\|\grad_xf\|})\mid x\in Y\setminus (\infty\cup \Crit(f))\}$ by the construction.
\begin{defini}
$\mathcal M_S^0(\pm f)= (\overline{M}_{\to}(f)+\overline{M}_{\to}(-f))(\Delta)
+\sum_{i,j}g_{ij}(\overline{\mathcal A_{q_i}}\times \overline{\mathcal D_{p_j}})(\Delta)
+\sum_{i,j}(-g_{ij})(\overline{\mathcal D_{p_j}}\times \overline{\mathcal A_{q_i}})(\Delta).$
\end{defini}
Let $\mathcal M_S(\pm f)$ be the extension of $\mathcal M_S^0(\pm f)$ to $C_2(Y)$.
\begin{lemma}\label{64}
$\mathcal M_S(\pm f)$ is a 4-cycle in $(C_2(Y),\partial C_2(Y))$.
\end{lemma}
\begin{proof}
Since
${\rm Im}(\partial (\overline{\mathcal A_{q_i}}\times\overline{\mathcal D_{p_j}})\to Y^2)=
\sum_k \partial_{ki}\overline{\mathcal A_{p_k}}\times \overline{\mathcal D_{p_j}}
+\sum_k\partial_{jk}\overline{\mathcal A_{q_i}}\times \overline{\mathcal D_{q_k}}$,
\begin{eqnarray*}
&&{\rm Im}(\sum_{i,j}g_{ij}\partial (\overline{\mathcal A_{q_i}}\times \overline{\mathcal D_{p_j}}\to Y^2))\\
&=&\sum_{i,j,k} g_{ij}\partial_{ki}\overline{\mathcal A_{p_k}}\times \overline{\mathcal D_{p_j}}
+\sum_{i,j,k}g_{ij}\partial_{jk}\overline{\mathcal A_{q_i}}\times \overline{\mathcal D_{q_k}}\\
&=&\sum_{i,j,k}\delta_{kj}\overline{\mathcal A_{p_k}}\times \overline{\mathcal D_{p_j}}
+\sum_{i,j,k}\delta_{ik}\overline{\mathcal A_{q_i}}\times \overline{\mathcal D_{q_k}}\\
&=&\sum_{j}\overline{\mathcal A_{p_j}}\times \overline{\mathcal D_{p_j}}
+\sum_{j}\overline{\mathcal A_{q_j}}\times \overline{\mathcal D_{q_j}}\\
&=& \partial \overline M_{\to}(f)\setminus \Delta .
\end{eqnarray*}
Therefore $\partial \mathcal M_S(\pm f)\setminus \partial C_2(Y)=\emptyset$.
\end{proof}
Under the identification $S\nu_{\Delta (Y\setminus\infty)}\cong ST(Y\setminus \infty)$, we have the following description.
\begin{lemma}\label{65}
$\partial \mathcal M_S(\pm f)\cap ST(Y\setminus  \infty)=\overline{\{ (x,\frac{\pm \grad_xf}{\|\grad_xf\|})\mid x\in Y\setminus (\infty\cup \Crit(f)\}}$.
\end{lemma}
\begin{proof}
Note that $(\overline{\mathcal A_{q_i}}\times \overline{\mathcal D_{p_j}})\cap\Delta=\overline{\mathcal A_{q_i}\cap \mathcal D_{p_j}}$.
By the definition of blow up, we have
$\partial \mathcal M_S(\pm f)\cap S\nu_{\Delta(Y\setminus \infty)}$
$$=\overline{\left\{ \left(x,\frac{\pm \grad_xf}{\|\grad_xf\|}\right)\right\}}+\sum_{i,j}g_{ij}\pi^{-1}(\overline{\mathcal A_{q_i}\cap \mathcal D_{p_j}}) +\sum_{i,j}(-g_{ij})\pi^{-1}(\overline{\mathcal D_{p_j}\cap \mathcal A_{q_i}})$$ where $\pi:STY\to Y$ is the projection.

Since $\sum_{i,j}g_{ij}\pi^{-1}(\overline{\mathcal A_{q_i}\cap \mathcal D_{p_j}})+\sum_{i,j}(-g_{ij})\pi^{-1}(\overline{\mathcal D_{p_j}\cap \mathcal A_{q_i}})=0$ as chains, we conclude the proof.
\end{proof}
\appendix
\section{Another proof of $\widetilde z_1(Y)=z_1^{\rm KKT}(Y)$.}
In this section we give a more direct proof of Preposition~\ref{77} in the case of $n=1$.
Remark that $\mathcal A_1(\emptyset)=\mathbb Q [\theta]$ and $\sharp \mathcal E_1=96$.
\begin{prop}[Proposition~\ref{77} in the case of $n=1$]\label{rationalprop}
$\widetilde z_1^{\rm anomaly}(\tau_Y^*\vec a)=\frac{1}{4}\sigma_{Y\setminus \infty}(\tau_Y)\delta_1.$
\end{prop}
To show this proposition we first prepare some notations and lemmas.
Let $\pi_1:F_X\to X$ be the tangent bundle along the fiber of $\pi_2: ST^vX\to X$.
Let $T^vX/TY$ be the real vector bundle over $X/Y$ obtained by collapsing $STY$ to a point using 
the framing $\tau_Y\cup\tau_{S^3}=\tau_Y|_{Y\setminus N(\infty;Y)}\cup\tau_{S^3}|_{N(\infty;Y)}$.
We define $F_{X/Y}, ST^vX/STY$ as same way.

Let $e(F_X;\tau_Y)\in H^2(ST^vX/STY)=H^2(ST^vX, STY)$ be the Euler class of $F_{X/Y}$
and let $p_1(F_X;\tau_Y)\in H^2(ST^vX/STY)=H^2(ST^vX, STY)$ be the 1st Pontrjagin class of $F_{X/Y}$.
By a standard argument, for example the Chern-Weil theory, we have $p_1(F_X;\tau_Y)=e_1(F_X;\tau_Y)^2$.
\begin{lemma}\label{lemma1}
$2[W(\tau_Y^*a)]=e(F_X;\tau_Y)\in H^2(ST^vX/STY)$.
\end{lemma} 
\begin{proof}
Let $\beta$ be the section of $T^vX$ such that $\beta|_{\partial X}=(\tau_Y\cup\tau_{S^3})^*a$ as Subsection~\ref{anomaly term}.
We define the map $f:ST^vX\to \RR$ by 
$$f(x)=\langle u,\beta(x)\rangle_{(T^vX)_x}$$
where $\langle,\rangle_{(T^vX)_x}$ is the standard inner product on $(T^vX)_x(\cong \RR^3)$.
We define the vector field $V\in \Gamma F_X$ by $V|_{(ST^vX)_x}=\grad (f|_{(ST^vX)_x})$ for any $x\in X$.
Thus $V$ is transverse to the zero section in $F_X$ and $V^{-1}(0)=c_0(\beta)$.
Thus the Poincar\'e dual of $(c_0(\beta),\partial c_0(\beta))$ represents $e(F_X;\tau_Y)$.
Since the closed 2-form $2W(\tau_Y^*a)$ represents the Poincar\'e dual of $(c_0(\beta),\partial c_0(\beta))$
and $W(\tau_Y^*a)|_{STY}=(\tau_Y\cup\tau_{S^3})^*\omega_{S^2}^a$,
we conclude the proof.  
\end{proof}
\begin{proof}[proof of Proposition~\ref{rationalprop}]
By the Lemma~\ref{lemma1}, we have
\begin{eqnarray*}
\int_{S_2(T^vX)}W(\tau_Y^*a)^3&=&\frac{1}{8}\int_{S_2(T^vX)}e(F_X;\tau_Y)^3\\
&=&\frac{1}{8}\int_{S_2(T^vX)}e(F_X;\tau_Y)p_1(F_X;\tau_Y)\\
&\stackrel{(*)}{=}&\frac{1}{8}\int_{S_2(T^vX)}e(F_X;\tau_Y)\pi_2^*p_1(TX;\tau_Y)\\
&=&\frac{1}{4}\int_{X}p_1(TX;\tau_Y)\\
&=&\frac{1}{4}\sigma_Y(\tau_Y\cup\tau_{S^3})+\frac{3}{4}\Sign X\\
&=&\frac{1}{4}\sigma_{Y\setminus \infty}(\tau_Y)+\frac{3}{4}\Sign X+\frac{1}{2}.
\end{eqnarray*}
The equation (*) is given by the following two relations: 
$\underline{\RR}\oplus F_X=\pi^*T^vX$ and $\underline{\RR}\oplus T^vX=TX$. 
Then we have
$$\widetilde z^{\rm anomaly}(\tau_Y^*\vec a)=96\int_{S_2(T^vX)}W(\tau_Y^*a)^3[\theta]-\mu_1\Sign X-c_1~~~~$$
$$~~~=\frac{96}{4}[\theta]\sigma_{Y\setminus \infty}(\tau_Y)+(72[\theta]-\mu_1)\Sign X-(c_1-48[\theta]).$$
Since this equation holds for any $\tau_Y$ and $X$, then we have $\mu_1=72[\theta]$, $c_1=48[\theta]$, $\delta_1=96[\theta]$.
Thus $\widetilde z^{\rm anomaly}_1(\tau_Y^*\vec a)=\frac{1}{4}\sigma_{Y\setminus \infty}(\tau_Y)\delta_1$.
\end{proof}
\bibliographystyle{alpha.bst}
\bibliography{tsbib}
\end{document}